\documentclass[10pt]{amsart}
\usepackage{geometry}
\usepackage{tikz-cd}
\usepackage{amssymb}
\usepackage{mathrsfs}
\usepackage[shortlabels]{enumitem}
\usepackage[colorlinks,linkcolor=black,citecolor=black,urlcolor=black]{hyperref}
\usepackage[colorinlistoftodos]{todonotes}
\usepackage{comment}

\numberwithin{equation}{section}
\newtheorem{thm}[subsection]{Theorem}
\newtheorem{cor}[subsection]{Corollary}
\newtheorem{lem}[subsection]{Lemma}
\newtheorem{prop}[subsection]{Proposition}
\theoremstyle{definition}
\newtheorem{df}[subsection]{Definition}
\newtheorem{rmk}[subsection]{Remark}
\newtheorem{exm}[subsection]{Example}
\newtheorem{const}[subsection]{Construction}

\newtheorem{conje}[subsection]{Conjecture}

\newcommand{\A}{\mathbb{A}}

\newcommand{\C}{\mathbb{C}}

\newcommand{\E}{\mathbb{E}}
\newcommand{\F}{\mathbb{F}}
\newcommand{\G}{\mathbb{G}}

\newcommand{\N}{\mathbb{N}}
\renewcommand{\P}{\mathbb{P}}
\newcommand{\Q}{\mathbb{Q}}
\newcommand{\R}{\mathbb{R}}
\newcommand{\unit}{\mathbf{1}}
\newcommand{\T}{\mathbb{T}}

\newcommand{\W}{\mathbb{W}}

\newcommand{\Z}{\mathbb{Z}}

\newcommand{\cB}{\mathcal{B}}
\newcommand{\cC}{\mathcal{C}}
\newcommand{\cD}{\mathcal{D}}

\newcommand{\cF}{\mathcal{F}}

\newcommand{\cO}{\mathcal{O}}
\newcommand{\cP}{\mathcal{P}}

\newcommand{\rD}{\mathrm{D}}

\newcommand{\rN}{\mathrm{N}}

\DeclareMathOperator{\Hom}{Hom}
\DeclareMathOperator{\Spec}{Spec}
\newcommand{\colim}{\mathop{\mathrm{colim}}}
\newcommand{\ul}{\underline}

\newcommand{\Fil}{\mathrm{Fil}}
\newcommand{\can}{\mathrm{can}}

\newcommand{\perf}{\mathrm{perf}}
\newcommand{\gp}{\mathrm{gp}}

\newcommand{\sat}{\mathrm{sat}}
\newcommand{\pt}{\mathrm{pt}}

\newcommand{\Cone}{\mathrm{Cone}}

\newcommand{\op}{\mathrm{op}}
\newcommand{\lSm}{\mathrm{lSm}}
\newcommand{\SmlSm}{\mathrm{SmlSm}}
\newcommand{\Poly}{\mathrm{Poly}}
\newcommand{\lPoly}{\mathrm{lPoly}}
\newcommand{\Ring}{\mathrm{Ring}}
\newcommand{\lRing}{\mathrm{lRing}}
\newcommand{\Pair}{\mathrm{Pair}}
\newcommand{\lPair}{\mathrm{lPair}}

\newcommand{\an}{\mathrm{an}}

\newcommand{\Set}{\mathrm{Set}}
\newcommand{\FinSet}{\mathrm{FinSet}}
\newcommand{\st}{\mathrm{st}}
\newcommand{\Fun}{\mathrm{Fun}}

\newcommand{\Spc}{\mathrm{Spc}}
\newcommand{\sInd}{\mathrm{sInd}}
\newcommand{\id}{\mathrm{id}}
\newcommand{\Env}{\mathrm{Env}}
\newcommand{\Mon}{\mathrm{Mon}}

\newcommand{\cofib}{\mathrm{cofib}}
\newcommand{\fib}{\mathrm{fib}}
\newcommand{\logSH}{\mathrm{logSH}}
\newcommand{\map}{\mathrm{map}}
\newcommand{\eff}{\mathrm{eff}}
\newcommand{\Coh}{\mathrm{Coh}}
\newcommand{\coker}{\mathrm{coker}}

\newcommand{\conj}{\mathrm{conj}}
\newcommand{\gr}{\mathrm{gr}}
\newcommand{\logTC}{\mathrm{logTC}}
\newcommand{\logTP}{\mathrm{logTP}}
\newcommand{\blogTC}{\mathbf{logTC}}
\newcommand{\blogTP}{\mathbf{logTP}}
\newcommand{\logTHH}{\mathrm{logTHH}}
\newcommand{\blogTHH}{\mathbf{logTHH}}
\newcommand{\BMS}{\mathrm{BMS}}

\newcommand{\TC}{\mathrm{TC}}
\newcommand{\trc}{\mathrm{trc}}
\newcommand{\syn}{\mathrm{syn}}

\newcommand{\Zar}{\mathrm{Zar}}
\newcommand{\ketale}{\mathrm{k\acute{e}t}}

\usepackage{relsize} 
\usepackage[bbgreekl]{mathbbol} 
\usepackage{amsfonts} 
\DeclareSymbolFontAlphabet{\mathbb}{AMSb} 
\DeclareSymbolFontAlphabet{\mathbbl}{bbold} 

\DeclareSymbolFontAlphabet{\mathbbl}{bbold} 
\newcommand{\cPrism}{{\widehat{\mathlarger{\mathbbl{\Delta}}}}}
\newcommand{\bPrism}{\mathbf{\widehat{\Delta}}}
\newcommand{\MZ}{\mathbf{M}\Z}

\begin{document}

\title{Log syntomic cohomology of truncated polynomials and coordinate axes}
\author{Doosung Park}
\address{Department of Mathematics and Informatics, University of Wuppertal, Germany}
\email{dpark@uni-wuppertal.de}
\author{Paul Arne {\O}stv{\ae}r}
\address{Department of Mathematics ``F. Enriques'', University of Milan, Italy \&
Department of Mathematics, University of Oslo, Norway}
\email{paul.oestvaer@unimi.it \& paularne@math.uio.no}
\date{\today}
\begin{abstract}
We study the logarithmic syntomic cohomology of fine and saturated log schemes and its realization in the logarithmic motivic stable homotopy category $\mathrm{logSH}(\mathrm{pt}_\mathbb{N})$ of a log point.  
We prove that logarithmic prismatic and syntomic cohomology satisfy saturated descent under the sole assumption that the log structure is free, and that the presheaves $\logTHH$, $\mathrm{logTC}$, $\cPrism$, and $\Z_p^\syn(i)$ are representable and $\square$-invariant in $\logSH_\ketale^\eff(\pt_\N)$.  
As an application, we compute $\Z_p^\syn(i)$ for the projective log coordinate axes $D$ in $\P^2$, obtaining
\[
\Z_p^\syn(i)(D)
\simeq
\Z_p^\syn(i)(k,\N)\oplus \Z_p^\syn(i-1)(k,\N)[-2]
\]
Moreover, we determine logarithmic topological cyclic homology for truncated polynomial and semistable examples, directly from the syntomic calculations.
\end{abstract}
\maketitle

\section{Introduction}

Prismatic cohomology, introduced by Bhatt and Scholze \cite{BS22}, interpolates between crystalline and de~Rham cohomology in mixed characteristic.  
Its success rests on a synthesis of derived and $p$-adic geometry: the prismatic site encodes infinitesimal thickenings equipped with divided powers and Frobenius-compatible connections, while the resulting cohomology controls topological cyclic homology and syntomic invariants.  
When one passes from smooth schemes to logarithmic geometry, these structures must be enriched to capture the behavior of normal crossings and semistable degenerations.  
The purpose of this paper is to develop this extension and to study logarithmic syntomic cohomology in the logarithmic motivic stable homotopy category $\logSH(\pt_\N)$ of a log point.

\vspace{0.8em}

Our approach proceeds through the formalism of \emph{log animated rings} introduced in~\cite{BLPO2}, which provides a derived enhancement of the category of log rings suited to prismatic constructions.  
Just as animated rings model derived infinitesimal thickenings, log animated rings encode both the derived and combinatorial data of logarithmic structures, allowing us to work uniformly with derived divided powers and connections.  
This framework clarifies the role of divided power envelopes and their Nygaard filtrations: they become natural functors on log animated rings, equipped with canonical connections whose horizontality conditions reflect the $p$-adic Cartier isomorphism.  
In particular, logarithmic prismatic cohomology $\cPrism_{(A,M)}$ may be described as the derived global sections of the log prismatic sheaf on the category of log quasi-syntomic log rings $(A,M)$, with coefficients in $\TC^-$, and carries a natural Nygaard filtration defined by divided Frobenius powers.

\vspace{0.8em}

A central conceptual development of this paper is the introduction of the notion of \emph{log Cartier smoothness}.  
In the classical prismatic setting, the condition of Cartier smoothness \cite{MR4264079} for a $\Z_p$-algebra ensures that its cotangent complex has $p$-complete Tor amplitude in degree $0$, thereby allowing the construction of the Nygaard filtration and the comparison with derived de~Rham cohomology.  
In the logarithmic context, semistable and toroidal situations require a parallel condition that controls both the infinitesimal and combinatorial contributions of the log structure.  
We therefore define a log ring $(A,M)$ to be \emph{log Cartier smooth} if
the cotangent complex \(L_{(A,M)/\F_p}\) has Tor amplitude concentrated in degree \(0\), 
and for every \(i\), the inverse Cartier map
\(
C^{-1} \colon \Omega_{(A,M)/\F_p}^i \to H^i(\Omega_{(A,M)/\F_p}^\bullet)
\)
is an isomorphism.
This notion provides the natural setting for logarithmic prismatic geometry: it guarantees the existence of the logarithmic Cartier isomorphism, the good behavior of the Nygaard filtration, and the compatibility between divided powers and connections.  
All the main constructions in this paper---including saturated descent, comparison with derived log de~Rham cohomology, and motivic representability in $\logSH(\pt_\N)$---are carried out in the class of log Cartier smooth objects.

\vspace{0.8em}

A recurring theme in this work is the role of de~Rham cohomology as the classical limit of prismatic and syntomic theories.  
The derived log de~Rham complex governs infinitesimal thickenings equipped with connections, and its $p$-adic completion appears as the associated graded object of the Nygaard filtration on log prismatic cohomology.  
In our setting, this perspective is essential: once the log structure $M$ is free, the Nygaard-completed log prismatic cohomology of $(A,M)$ can be identified with the $p$-completed derived de~Rham cohomology of a suitable $\Z_p$-lift.  
Thus, de~Rham cohomology provides the bridge between divided powers and connections on one hand, and motivic syntomic phenomena on the other, revealing the geometric content of logarithmic prismatic cohomology.

\vspace{0.8em}

A key ingredient in our analysis is the technique of \emph{saturated descent}, introduced in~\cite{BLMP}.  
It asserts that the Nygaard-completed log prismatic cohomology of $(A,M)$ can be reconstructed from the prismatic cohomology of
animated rings, provided suitable exactness conditions hold.  
We show that this descent statement remains valid whenever the log structure $M$ is free, 
thereby simplifying the hypotheses of~\cite[Theorem~4.18]{BLMP}.  
This refinement enables a direct comparison between logarithmic prismatic cohomology and the $p$-completed derived log de~Rham cohomology of an appropriate lift to $\Z_p$, recovering the classical logarithmic de~Rham complex in characteristic~$p$ as the associated graded of the Nygaard filtration.

\vspace{0.8em}

The main object of study in this paper is the \emph{logarithmic syntomic complex} $\Z_p^\syn(i)(A,M)$, 
defined in terms of the Nygaard filtration on log prismatic cohomology by
\[
\Z_p^\syn(i)(A,M)
:=
\fib(
\Fil_{\rN}^{\geq i}\cPrism_{(A,M)}
\xrightarrow{\varphi_p-\mathrm{can}}
\cPrism_{(A,M)})
\]
This construction generalizes the syntomic complexes of~\cite{BMS19} to logarithmic geometry and yields a cohomology theory that encodes both Frobenius and connection data in the log setting.  
We prove that $\Z_p^\syn(i)$, which is functorial in log quasi-syntomic morphisms, 
satisfies saturated descent for free monoids, 
and is representable in the logarithmic motivic category $\logSH(\pt_\N)$.  
In particular, we show that
\[
\Z_p^\syn(i),
\quad
\cPrism,
\quad
\logTHH,
\quad
\logTC^-,
\quad
\logTP,
\quad
\text{and }
\logTC
\in
\logSH_\ketale^\eff(\pt_\N),
\]
so that all major logarithmic cohomology theories admit canonical motivic realizations.

\vspace{0.8em}

From the syntomic viewpoint, logarithmic topological cyclic homology is controlled by
$\Z_p^\syn(i)$.  
For every log ring $(A,M)$ there is a 
\emph{logarithmic trace map}
\[
\trc_{\log}\colon K(A,M) \longrightarrow \logTC(A,M)
\]
obtained by left Kan extension from the regular log regular case
\cite[Theorem 8.6.3]{BPO2} and \cite[Theorem F]{logSHF1}.
The map $\trc_{\log}$ provides a natural bridge between logarithmic $K$-theory and prismatic phenomena.  
However, we do not attempt to compute logarithmic $K$-theory itself in this work.  
Instead, we  
determine the homotopy of $\logTC$ in concrete cases, 
such as truncated polynomial and semistable examples, directly from syntomic calculations.

\vspace{0.8em}

As an explicit illustration, let $k$ be a perfect $\F_p$-algebra and consider the log ring $(k[x]/x^e,\N)$ with $1\mapsto x$.  
We compute
\[
\Z_p^\syn(i)(k[x]/x^e,\N)
\simeq
\begin{cases}
W(k) \oplus W(k)[-1] & i = 0 \\[2mm]
\W_{ei-1}(k)[-1] \oplus W(k)[-1] \oplus W(k)[-2] & i = 1 \\[2mm]
\W_{ei-1}(k)[-1] & i > 1 
\end{cases}
\]
where \(W(k)\) denotes the ring of Witt vectors and \(\W_j(k)\) denotes the ring of big Witt vectors \cite{MR3316757} truncated at \(\{1,\ldots,j\}\).
From which it follows that
\[
\pi_i(\logTC(k[x]/x^e,\N)_p^\wedge)
\cong
\begin{cases}
W(k) & i = -1 \\[2mm]
W(k)\oplus W(k) & i = 0 \\[2mm]
W(k)\oplus \W_{e-1}(k) & i = 1 \\[2mm]
\W_{em-1}(k) & i = 2m-1\ge 3 \\[2mm]
0 & \text{otherwise}
\end{cases}
\]
These formulas extend the classical truncated polynomial results of Hesselholt and Madsen \cite{zbMATH01058847} in the logarithmic setting, confirming that the syntomic theory faithfully detects the structure of logarithmic topological cyclic homology.

\vspace{0.8em}

In the second part of the paper we study the representability of these cohomology theories in the effective logarithmic motivic stable homotopy category $\logSH^\eff(\pt_\N)$.  
We show that the presheaves
\[
R\Gamma_\Zar(-,\wedge^i L_{-/k}),
\quad
\cPrism,
\quad
\Fil_\rN^{\ge i}\cPrism,
\quad
\text{and }
\Z_p^\syn(i)
\]
are log \'etale sheaves and $\square$-invariant on the category of log smooth fs log schemes over $\pt_\N = \Spec(k,\N)$ for a perfect $\F_p$-algebra $k$.  
Consequently, these theories are representable by canonical objects in $\logSH_\ketale^\eff(\pt_\N)$, extending the motivic representability of prismatic and syntomic theories established in~\cite{BMS19} and~\cite{BPO2}.  
The resulting logarithmic motivic spectra satisfy Bott periodicity; 
e.g., for logarithmic cyclic homology, we have
\[
\map_{\logSH(\pt_\N)}(\Sigma_{\P^1}^\infty X_+,\blogTC)
\simeq
\logTC(X), \qquad
\blogTC\simeq \Sigma^{2,1}\blogTC,
\]
mirroring the periodic behavior of algebraic $K$-theory in the motivic setting.

\vspace{0.8em}

Finally, we illustrate these constructions by an explicit computation for the projective log coordinate axes
\[
D = (\P^2,\ul{D}) \times_{\P^2} \ul{D}, \qquad
\ul{D} = \{[x:y:z] : x=0 \text{ or } y=0\},
\]
a basic example of a semistable fs log scheme.  
Using toric geometry and dividing covers, we show an explicit decomposition
\[
\Z_p^\syn(i)(D)
\simeq
\Z_p^\syn(i)(k,\N)
\oplus
\Z_p^\syn(i-1)(k,\N)[-2]
\]
This identifies the logarithmic syntomic cohomology of $D$ with successive extensions of the syntomic cohomology of the logarithmic point, and demonstrates concretely how divided powers, connections, and de~Rham-type structures are reflected in the motivic category.

\vspace{0.8em}

The results presented here show that logarithmic prismatic and syntomic cohomology admit a natural interpretation in terms of animated log geometry and motivic homotopy theory.  
In particular, the interplay between divided powers, connections, log Cartier smoothness, and de~Rham cohomology becomes an instance of the motivic relationship between prismatic and syntomic phenomena.

\subsection*{Acknowledgements}
This work was supported by the RCN Project no.~312472 Equations in Motivic Homotopy, the European Commission — Horizon-MSCA-PF-2022 Motivic integral $p$-adic cohomologies, and the DFG-funded research training group GRK 2240: Algebro-Geometric Methods in Algebra, Arithmetic and Topology.

\section{Animation of log pairs}

This section develops the categorical framework underlying the notion of \emph{animated log pairs}.  
It begins by reviewing the $1$-sifted ind-category $\sInd(\mathcal{C})$ and its universal property, relating it to the category $\mathcal{P}_{\Sigma,1}(\mathcal{C})$ of product-preserving presheaves.  
A general equivalence $\sInd(\cC)\simeq\cD$ is established when $\cC$ is a full subcategory of compact projective objects generating a cocomplete category $\cD$ under colimits.  
Several examples illustrate this framework, including the equivalences $\sInd(\FinSet)\simeq\Set$ and $\sInd(\Poly_R)\simeq\Ring_R$.  
Extending to logarithmic geometry, this section proves $\sInd(\lPoly_{(R,P)})\simeq\lRing_{(R,P)}$ for a log ring $(R,P)$ and analyzes the behavior of categories of pairs and log pairs with respect to $1$-sifted colimits.  
It culminates in a fully faithful embedding of $\lPair_{(R,P)}$ into $\sInd(\lPair_{(R,P)}^\st)$ and defines the $\infty$-category of \emph{animated log pairs}
\[
\lPair_{(R,P)}^\an := \cP_\Sigma(\lPair_{(R,P)}^\st)
\]
which provides the natural setting for derived and animated constructions in logarithmic geometry.
\vspace{0.1in}

For a category $\mathcal{C}$, let $\sInd(\mathcal{C})$ denote the $1$-sifted ind-category introduced in \cite[\S5.1.1]{MR4681144}.  
This is the full subcategory of presheaves of sets on $\mathcal{C}$ generated under $1$-sifted colimits (called sifted colimits in \cite{MR2720191}) by the Yoneda image of $\mathcal{C}$.  
By \cite[Theorem 2.1]{MR2720191}, a functor preserves $1$-sifted colimits if and only if it preserves filtered colimits and reflexive coequalizers, provided the source admits finite colimits.

If $\mathcal{C}$ has finite coproducts, consider the category $\mathcal{P}_{\Sigma,1}(\mathcal{C})$ from \cite[Notation A.10]{MR4851406}, consisting of functors $\mathcal{C}^{\op} \to \Set$ preserving finite products.  
By \cite[Proposition A.11(4)]{MR4851406} and \cite[Lemma 3.8]{MR2640203}, there is a natural equivalence
\[
\sInd(\mathcal{C}) \simeq \mathcal{P}_{\Sigma,1}(\mathcal{C})
\]
and each object of $\sInd(\mathcal{C})$ is a reflexive coequalizer of representable presheaves.  
Furthermore, \cite[Proposition A.14]{MR4851406} gives the following universal property.  
If $\mathcal{D}$ admits $1$-sifted colimits, then the functor
\begin{equation}\label{ani.1.2}
\Fun_\Sigma(\sInd(\mathcal{C}), \mathcal{D}) \xrightarrow{\simeq} \Fun(\mathcal{C}, \mathcal{D})
\end{equation}
induced by the Yoneda embedding $j\colon \mathcal{C} \hookrightarrow \sInd(\mathcal{C})$ is an equivalence, where $\Fun_\Sigma$ denotes the category of functors preserving $1$-sifted colimits.  
Moreover, a functor $g\colon \sInd(\mathcal{C}) \to \mathcal{D}$ preserves colimits if and only if $gj$ preserves finite coproducts.  
If $F\colon \mathcal{C} \to \mathcal{D}$ preserves finite coproducts, the induced functor $\sInd(\mathcal{C}) \to \mathcal{D}$ is left adjoint to $X \mapsto \Hom_{\mathcal{D}}(F(-),X)$ \cite[Lemma 2.1]{MR4851406}.

Recall from \cite[Definition 5.5.8.8]{HTT} that $\mathcal{P}_\Sigma(\mathcal{C})$ is the $\infty$-category of functors $\mathcal{C}^{\op} \to \Spc$ preserving finite products, where $\Spc$ denotes the $\infty$-category of spaces.  
Note that $\mathcal{P}_{\Sigma,1}(\mathcal{C})$ is a full subcategory of $\mathcal{P}_\Sigma(\mathcal{C})$.  
An object $X$ of $\mathcal{C}$ is called \emph{compact projective} if $\Hom_{\mathcal{C}}(X,-)$ preserves $1$-sifted colimits.

\begin{prop}\label{ani.1}
Let $\cD$ be a cocomplete category and $\cC \subset \cD$ a full subcategory closed under finite coproducts, whose objects are compact projective.  
If $\cC$ generates $\cD$ under colimits, then there is an equivalence
\[
\sInd(\cC) \simeq \cD
\]
\end{prop}

\begin{proof}
We follow \cite[Footnotes 17, 19]{MR4681144}.  
Let $F\colon \sInd(\cC) \to \cD$ be the functor from \eqref{ani.1.2}.  
It suffices to show that $F$ is fully faithful, i.e.
\begin{equation}\label{ani.1.1}
\Hom_{\sInd(\cC)}(X,Y) \cong \Hom_{\cD}(F(X),F(Y))
\end{equation}
for all $X,Y \in \sInd(\cC)$.  
The isomorphism holds when $X,Y \in \cC$; since every object of $\cC$ is compact and projective, it extends to all $X,Y$.

Because $\sInd(\cC)$ is equivalent to the category of functors $\cC^{\op}\to\Set$ preserving finite coproducts, it admits finite coproducts, and thus coproducts via filtered colimits.  
Every colimit is a reflexive coequalizer of coproducts, so $\sInd(\cC)$ admits all colimits.  
Having shown $F$ is fully faithful, the result follows.
\end{proof}

\begin{exm}\label{ani.4}
The full subcategory of compact projective objects in $\Set$ is $\FinSet$, the category of finite sets.  
By Proposition~\ref{ani.1},
\[
\sInd(\FinSet) \simeq \Set
\]
\end{exm}

\begin{exm}\label{ani.2}
For $n\ge1$, let $\Set^n$ (resp.\ $\FinSet^n$) denote the Cartesian product of $n$ copies of $\Set$ (resp.\ $\FinSet$).  
If $\cD$ admits colimits and
\[
F\colon \Set^n \rightleftarrows \cD : U
\]
is an adjoint pair with $U$ conservative and preserving $1$-sifted colimits, then by \cite[Proof of Cor.~4.7.3.18]{HA}, the subcategory $\cC := F(\FinSet^n)$ satisfies Proposition~\ref{ani.1}, giving
\[
\sInd(\cC) \simeq \cD
\]
\end{exm}

\begin{exm}
Let $R$ be a ring.  
By \cite[Construction 2.1]{BMS19},
\[
\sInd(\Poly_R) \simeq \Ring_R
\]
where $\Poly_R$ is the category of finitely generated polynomial $R$-algebras.  
This corresponds to Example~\ref{ani.2} with $\cD=\Ring_R$ and $n=1$.  
The forgetful functor $\Ring_R\to\Set$ preserves $1$-sifted colimits \cite[Example 5.1.3]{MR4681144}.
\end{exm}

\begin{exm}\label{ani.3}
Let $P$ be a monoid, and let $\Mon_P$ denote the category of monoid homomorphisms $P\to M$.  
The conservative forgetful functor $\Mon_P\to\Set$ has a left adjoint
\[
\Set \to \Mon_P, \quad S\mapsto P\oplus \N^S
\]
It preserves reflexive coequalizers.  
Given
\[
\begin{tikzcd}
M \ar[r,"\eta"] & N \ar[r,shift left=0.75ex,"\theta_1"] \ar[r,shift right=0.75ex,"\theta_2"'] & M
\end{tikzcd}
\]
the coequalizer of $\theta_1,\theta_2$ is $M$ modulo the congruence generated by $\theta_1(y)\sim\theta_2(y)$ for all $y\in N$.  
As congruences satisfy $a+c\sim b+c$ whenever $a\sim b$, this relation is a congruence (see \cite[\S I.1.1]{Ogu}).  
Hence $\Mon_P\to\Set$ preserves reflexive coequalizers and therefore $1$-sifted colimits.
\end{exm}

\begin{exm}\label{ani.5}
Let $(R,P)$ be a log ring, i.e.\ a ring $R$ with a monoid $P$ and a homomorphism $P\to R$.  
Let $\lRing_{(R,P)}$ denote the category of log $(R,P)$-algebras.  
The forgetful functor $\lRing_{(R,P)}\to\Set^2$, $(A,M)\mapsto(A,M)$, has a left adjoint
\[
\Set^2 \to \lRing_{(R,P)}, \quad (S,T)\mapsto (R[\N^S\oplus \N^T], P\oplus\N^T)
\]
Let $\lPoly_{(R,P)}$ be the essential image of $\FinSet^2$ under this functor.  
The forgetful functor $\lRing_{(R,P)}\to\Ring_R\times\Mon_P$ is conservative and preserves $1$-sifted colimits (by \cite[Example 5.1.3]{MR4681144} and Example~\ref{ani.3}).  
Thus the same holds for $\lRing_{(R,P)}\to\Set^2$, and Example~\ref{ani.2} gives:
\end{exm}

\begin{prop}\label{ani.6}
For any log ring $(R,P)$,
\[
\sInd(\lPoly_{(R,P)}) \simeq \lRing_{(R,P)}
\]
\end{prop}

\begin{exm}\label{ani.7}
Let $R$ be a ring, and let $\Pair_R$ denote the category of pairs $(A,I)$ with $A$ an $R$-algebra and $I\subset A$ an ideal.  
The forgetful functor $\Pair_R\to\Set^2$, $(A,I)\mapsto(A,I)$, has a left adjoint
\[
\Set^2\to\Pair_R, \quad (S,T)\mapsto(\R[\N^S\oplus\N^T],(T))
\]
where $(T)$ is the ideal generated by the image of $T$.  
Let $\Pair_R^\st$ be the essential image of $\FinSet^2$.  
Mao \cite[Remark 3.14]{MR4851406} showed that
\[
\sInd(\Pair_R^\st)\not\simeq\Pair_R
\]
so $\Pair_R\to\Set^2$ does not preserve $1$-sifted colimits.  
However, \cite[Lemma 3.12]{MR4851406} implies $\Pair_R$ is a full subcategory of $\sInd(\Pair_R^\st)$.  
Thus we may view $\Pair_R$ as a full subcategory of the $\infty$-category of \emph{animated pairs}
\[
\Pair_R^\an := \cP_\Sigma(\Pair_R^\st)
\]
\end{exm}

\begin{exm}\label{ani.8}
Let $(R,P)$ be a log ring, and $\lPair_{(R,P)}$ the category of triples $(A,M,I)$ where $(A,M)$ is a log $(R,P)$-algebra and $I\subset A$ an ideal.  
The forgetful functor $\lPair_R\to\Set^3$, $(A,M,I)\mapsto(A,M,I)$, has a left adjoint
\[
\Set^3\to\lPair_{(R,P)}, \quad (S,T,U)\mapsto (R[\N^S\oplus\N^T\oplus\N^U], P\oplus\N^T,(\N^U))
\]
Let $\lPair_{(R,P)}^\st$ be the essential image of $\FinSet^3$.  
As in Example~\ref{ani.7}, the forgetful functor $\lPair_{(R,P)}\to\Set^3$ does not preserve $1$-sifted colimits, but:
\end{exm}

\begin{prop}\label{ani.9}
For any log ring $(R,P)$, the functor
\[
\lPair_{(R,P)} \to \sInd(\lPair_{(R,P)}^\st), \quad (A,M,I)\mapsto \Hom_{\lPair_{(R,P)}}(-,(A,M,I))|_{\lPair_{(R,P)}^\st}
\]
is fully faithful.
\end{prop}

\begin{proof}
Let $(A,M,I),(A',M',I')\in\lPair_{(R,P)}$.  
A natural transformation
\[
F\colon \Hom_{\lPair_{(R,P)}}(-,(A,M,I))|_{\lPair_{(R,P)}^\st} \to \Hom_{\lPair_{(R,P)}}(-,(A',M',I'))|_{\lPair_{(R,P)}^\st}
\]
arises from a morphism $(A,M,I)\to(A',M',I')$.  
Precomposing with $(\FinSet^3)^\op\to(\lPair_{(R,P)}^\st)^\op$ gives a transformation
\[
\Hom_{\Set^3}(-,(A,M,I))|_{\FinSet^3} \to \Hom_{\Set^3}(-,(A',M',I'))|_{\FinSet^3}
\]
which corresponds to some $f\colon (A,M,I)\to(A',M',I')$ in $\Set^3$.  
Composing with $(\lPoly_{(R,P)})^\op\to(\lPair_{(R,P)}^\st)^\op$ yields a map of log algebras $g\colon(A,M)\to(A',M')$ (Proposition~\ref{ani.6}), compatible with the $\Set^3$ part of $f$.  
Similarly, precomposition with $(\Pair_{(R,P)}^\st)^\op$ gives $h\colon(A,I)\to(A',I')$ (\cite[Lemma 3.12]{MR4851406}).  
These agree on overlaps, hence $f$ defines a morphism in $\lPair_{(R,P)}$.
\end{proof}

\begin{df}\label{ani.10}
For a log ring $(R,P)$, set
\[
\lPair_{(R,P)}^\an := \cP_\Sigma(\lPair_{(R,P)}^\st)
\]
By Proposition~\ref{ani.9}, $\lPair_{(R,P)}^\an$ contains $\lPair_{(R,P)}$ as a full subcategory.
\end{df}

\section{Saturated Descent for the Free Monoid Case}
\label{section:Saturated Descent for the Free Monoid Case}

This section extends the saturated descent method of \cite[Theorem~4.18]{BLMP} to the case where the log structure \(M\) is free. 
After introducing the colimit perfection \(M_\perf\) of a saturated monoid and its variants \(M_{\perf(n)}\), it is shown that for a log \(R\)-algebra \((A,M)\) with free \(M\) and bounded \(p^\infty\)-torsion, the logarithmic cotangent complex admits a descent formula analogous to that for rings. 
A key step establishes that, under freeness, the quasi-syntomic and log quasi-syntomic conditions on \(A\) and \((A,M)\) coincide. 
Using this, the log prismatic cohomology \(\cPrism_{(A,M)}\) is expressed as the limit of prismatic cohomologies of the system \(\{A \widehat{\otimes}_{\Z_p\langle M\rangle} \Z_p\langle M_{\perf(n)}\rangle\}_n\). 
This equivalence implies that the log Nygaard filtration and syntomic cohomology satisfy the same descent properties. 
Finally, when \(A\) is a \(p\)-torsion-free quasi-syntomic \(\delta\)-ring admitting a map \(\Z_p\langle M\rangle \to A\), a natural equivalence
\[
\cPrism_{(A/p,M)} \simeq (L\Omega_{(A,M)/\Z_p})_p^\wedge
\]
relates log prismatic and derived de Rham cohomology. 
Together, these results provide a simplified and more general proof of saturated descent for log rings with free monoid structures.
\vspace{0.1in}

Let \((A,M)\) be a log quasi-syntomic log ring in the sense of \cite[\S 4.4]{BLPO2} or \cite[Definition 3.2]{2306.00364}, so that \(A\) has bounded \(p^\infty\)-torsion, \(M\) is integral, and \(L_{(A,M)/\Z_p}\) has \(p\)-complete Tor amplitude in \([-1,0]\). 
By \cite[\S 7.2]{BLPO}, the log prismatic cohomology \(\cPrism_{(A,M)}\) is the global section of the log quasi-syntomic sheaf \cite[Definition 4.5]{BLPO} on the category of log quasi-syntomic log \((A,M)\)-algebras, whose value at a log quasi-regular semiperfectoid ring \cite[Definition 4.7]{BLPO} \(S\) is \(\pi_0 \TC^-(S)_p^\wedge\).

\begin{df}
\label{sat.1}
For a saturated monoid \(M\), the \emph{colimit perfection of \(M\)} is defined as
\[
M_\perf := \colim(M \xrightarrow{p} M \xrightarrow{p} \cdots)
\]
For any integer \(n \geq 0\), set
\[
M_{\perf(n)} := M_\perf \oplus_M^\sat \cdots \oplus_M^\sat M_\perf
\text{ (with \(n+1\) copies of \(M_\perf\))}
\]
Here \(\oplus^\sat\) denotes the coproduct in the category of saturated monoids. 
From \cite[Lemma 4.4]{BLMP}, for each integer \(n \geq 1\) there is a natural isomorphism of monoids
\begin{equation}
M_{\perf(n)} \cong M_\perf \oplus (M_\perf^\gp / M^\gp)^{\oplus (n-1)}
\end{equation}
\end{df}

Let \(\Z_p\langle M\rangle\) denote the \(p\)-completion of \(\Z_p[M]\).

\begin{prop}
\label{sat.2}
Let \((A,M)\) be a log \(R\)-algebra, where \(R\) is a ring. 
If \(A\) and \(R\) have bounded \(p^\infty\)-torsion and \(M\) is free, there is a natural equivalence of complexes
\[
(\wedge^i L_{(A,M)/R})_p^\wedge
\to
\lim_{n\in \Delta^\op}
(\wedge^i L_{A\widehat{\otimes}_{\Z_p\langle M\rangle} \Z_p\langle M_{\perf(n)}\rangle /R})_p^\wedge
\]
for every integer \(i \geq 0\).
\end{prop}

\begin{proof}
This is a special case of \cite[Theorem 4.12]{BLMP}. 
Since \(M\) is free, \(\Z[M] \to \Z[M_\perf]\) is flat, implying that \(A \widehat{\otimes}_{\Z_p\langle M\rangle}^L \Z_p\langle M_\perf\rangle\) is discrete.
\end{proof}

We next introduce notation for filtrations. 
Let \(\cC\) be an \(\infty\)-category and \(X \in \cC\). 
A \emph{filtration on \(X\)} is a functor \(\Fil^{\geq \bullet} X \colon \Z^\op \to \cC\) equipped with a morphism \(\colim_n \Fil^{\geq n}X \to X\), where \(\Z^\op\) has objects \(\Z\) and morphisms \(\Hom_{\Z^\op}(m,n)=\ast\) if \(m \geq n\) and is empty otherwise. 
The filtration \(\Fil^{\geq \bullet} X\) is \emph{complete} (resp.\ \emph{exhaustive}) if \(\lim_n \Fil^{\geq n}X \simeq 0\) (resp.\ if \(\colim_n \Fil^{\geq n}X \to X\) is an equivalence).

\begin{const}
\label{sat.6}
Let \(R\) be a ring and \((R,0)\) the log ring with trivial monoid. 
For \((A,M) \in \lPoly_{(R,0)}\), let \(\Omega_{(A,M)/R}^\bullet\) be the filtered \(\E_\infty\)-ring defined by the Hodge filtration \(\Omega_{(A,M)/R}^{\geq n}\). 
For a log \(R\)-algebra \((A,M)\), the filtered \(\E_\infty\)-ring \(L\Omega_{(A,M)/R}\) obtained by left Kan extension carries the induced \emph{Hodge filtration}.
\end{const}

\begin{const}
\label{sat.7}
Let \((A,M)\) be a log \(\Z_p\)-algebra. 
In Theorem \ref{sat.8} we also consider the filtration on \(L\Omega_{(A,M)/\Z_p}\) given by the tensor product of the Hodge filtration and the \(p\)-adic filtration on \(\Z_p\).
\end{const}

\begin{prop}
\label{sat.5}
Let \((A,M)\) be a log \(\Z_p\)-algebra with \(M\) free. 
Then \(A\) is quasi-syntomic if and only if \((A,M)\) is log quasi-syntomic.
\end{prop}

\begin{proof}
From the transitivity sequence
\[
L_{A/\Z_p} \to L_{(A,M)/\Z_p} \to L_{(A,M)/A}
\]
it suffices to show that \(L_{(A,M)/A}\) has \(p\)-complete Tor amplitude in \([-1,0]\). 
By the description of \(L_{(A,M)/A}\) in \cite[(3.5)]{BLPO}, there is a fiber sequence
\[
A \otimes_{\Z[M]} L_{\Z[M]/\Z} \to A \otimes_{\Z[M]} L_{(\Z[M],M)/\Z} \to L_{(A,M)/A}
\]
Both \(A \otimes_{\Z[M]} L_{\Z[M]/\Z}\) and \(A \otimes_{\Z[M]} L_{(\Z[M],M)/\Z}\) have \(p\)-complete Tor amplitude concentrated in degree \(0\); the freeness of \(M\) ensures this for the first term.
\end{proof}

For a log quasi-syntomic log ring \((A,M)\), the Nygaard filtration \(\Fil_\rN^{\geq \bullet}\) on \(\cPrism_{(A,M)}\) is defined in \cite[\S 7.2]{BLPO2} and extends the Nygaard filtration on the Nygaard-completed prismatic cohomology of quasi-syntomic rings \cite[Construction 7.7]{BMS19}. 
The log syntomic cohomology of \((A,M)\) is given by
\[
\Z_p^\syn(i)(A,M)
:=
\fib(
\Fil_{\rN}^{\geq i}\cPrism_{(A,M)}
\xrightarrow{\varphi_p-\mathrm{can}}
\cPrism_{(A,M)})
\]
where \(\varphi_p\) is the cyclotomic Frobenius and \(\mathrm{can}\) the canonical morphism.

\begin{prop}
\label{sat.3}
Let \((A,M)\) be a saturated log quasi-syntomic log ring with \(M\) free. 
Then there is a natural equivalence of filtered \(\E_\infty\)-rings
\[
\cPrism_{(A,M)} \simeq \lim_{n \in \Delta^\op} \cPrism_{A \widehat{\otimes}_{\Z\langle M\rangle} \Z\langle M_{\perf(n)}\rangle}
\]
\end{prop}

\begin{proof}
By Proposition \ref{sat.5}, \(A\) is quasi-syntomic. 
Consider the natural morphism of filtered \(\E_\infty\)-rings
\[
\cPrism_{(A, M)} \to \lim_{n \in \Delta^{\text{op}}} \cPrism_{(A \widehat{\otimes}_{\Z\langle M \rangle} \Z\langle M_{\text{perf}(n)} \rangle, M_{\text{perf}(n)})}
\]
To show this is an equivalence, note that the natural map
\[
\lim_{n \in \Delta^{\text{op}}} \cPrism_{(A \widehat{\otimes}_{\Z\langle M \rangle} \Z\langle M_{\text{perf}(n)} \rangle)} \to \lim_{n \in \Delta^{\text{op}}} \cPrism_{(A \widehat{\otimes}_{\Z\langle M \rangle} \Z\langle M_{\text{perf}(n)} \rangle, M_{\text{perf}(n)})}
\]
is an equivalence by \cite[Proposition 4.14]{BLMP}. 
Let \(\Z_p^{\text{cyc}} := \Z_p[\mu_{p^\infty}]_p^{\wedge}\) \cite[Example 3.6]{BMS18}. 
Since \(\Z_p^{\text{cyc}}\) is quasi-syntomic and \(\Z_p \to \Z_p^{\text{cyc}}\) is \(p\)-completely faithfully flat, it is a quasi-syntomic cover. 
Applying quasi-syntomic descent for \(A \to A \widehat{\otimes}_{\Z_p} \Z_p^{\text{cyc}}\) reduces the claim to the case where \(A\) is an \(R\)-algebra for a perfectoid ring \(R\). 
The result then follows from Proposition \ref{sat.2} and \cite[Proposition 7.4]{BLPO2}.
\end{proof}

\begin{rmk}
\label{sat.4}
We compute \(\Z_p^\syn(i)(k[x]/x^e,\N)\) in Theorem \ref{logfat.1} for a perfect \(\F_p\)-algebra \(k\) and integers \(e,i \geq 1\), where the log structure map \(\N \to k[x]/x^e\) sends \(n \mapsto x^n\). 
Note that \((k[x]/x^e,\N)\) does not satisfy the assumption of \cite[Theorem 4.18]{BLMP}, since the induced map \(k[\N] \to k[x]/x^e\) is not flat. 
This motivates the alternative assumption in Proposition \ref{sat.3}. 
Moreover, Proposition \ref{sat.3} implies \cite[Theorem 4.30]{BLMP}\ only for the case of \(\SmlSm_k\) rather than \(\lSm_k\).
\end{rmk}

From \cite[Example 2.4(3)]{2007.14037}, every monoid \(M\) gives rise to a natural \(\delta\)-ring structure on \(\Z_{(p)}[M]\), and hence on \(\Z_p\langle M \rangle\).

\begin{thm}
\label{sat.8}
Let \((A,M)\) be a saturated log quasi-syntomic log ring such that \(A\) is a \(p\)-torsion-free quasi-syntomic \(\delta\)-ring, \(M\) is free, and there exists a \(\delta\)-ring homomorphism \(\Z_p\langle M\rangle \to A\). 
Then there is a natural equivalence of filtered \(\E_\infty\)-rings
\[
\cPrism_{(A/p,M)} \simeq (L\Omega_{(A,M)/\Z_p})_p^\wedge
\]
where \(\cPrism_{(A/p,M)}\) carries the Nygaard filtration and \(L\Omega_{(A,M)/\Z_p}\) the filtration from Construction \ref{sat.7}.
\end{thm}

\begin{proof}
By Propositions \ref{sat.2} and \ref{sat.3}, it suffices to construct a natural equivalence of filtered \(\E_\infty\)-rings
\[
\cPrism_{A/p \widehat{\otimes}_{\Z_p\langle M \rangle} \Z_p\langle M_{\perf(n)}\rangle}
\simeq
(L\Omega_{A \widehat{\otimes}_{\Z_p\langle M \rangle} \Z_p\langle M_{\perf(n)}\rangle/\Z_p})_p^\wedge
\]
for each integer \(n \geq 0\). 
Since \(\Z[M] \to \Z[M_{\perf(n)}]\) is flat, \(A \widehat{\otimes}_{\Z_p\langle M\rangle} \Z_p\langle M_{\perf(n)}\rangle\) is \(p\)-torsion-free. 
Moreover, both \(\Z_p\langle M\rangle \to A\) and \(\Z_p\langle M\rangle \to \Z_p\langle M_{\perf(n)}\rangle\) are \(\delta\)-ring maps, so \(A \widehat{\otimes}_{\Z_p\langle M\rangle} \Z_p\langle M_{\perf(n)}\rangle\) is itself a \(\delta\)-ring by \cite[Remark 2.7]{BS22}. 
The result follows from \cite[Remark 8.11]{zbMATH07729746}.
\end{proof}

\section{Divided power algebras}
\label{section:Divided power algebras}

In this section, 
we review and extend the theory of divided power (DP) algebras and their derived analogues. 
Starting from the classical DP envelope functor \(\Env: (A,I) \mapsto (D_A(I),\overline I,\overline\gamma)\), we recall its filtration by divided powers and its independence from the base ring. 
The discussion then passes to the animated setting: Mao’s derived divided power envelope \(\Env^\an\) and the associated filtered animated ring \(LD_I(A)\), obtained by left Kan extension from the classical construction. 
A conjugate filtration \(\Fil_\conj^{\ge \bullet}\) is introduced, with graded pieces expressed as
\[
\gr_\conj^{-m} LD_A(I) \simeq (\wedge_{A/I}^m L_{(A/I)/A})^{(1)}[-m],
\]
and, when \(L_{(A/I)/A}\) has Tor amplitude in degree \(-1\), this specializes to
\[
\gr_\conj^{-m} LD_A(I) \simeq (\Gamma_{A/I}^m(I/I^2))^{(1)}.
\]
It is shown that if \(LD_A(I)\) is discrete, then \(LD_I(A)\simeq D_I(A)\) as animated (and, under mild conditions, filtered) rings.  
Mathew’s criterion (Proposition~\ref{comp.11}) ensures discreteness and torsion-freeness of \(LD_A(I)\) under flatness and Tor-amplitude hypotheses, implying \(LD_A(I)\simeq D_A(I)\).  
Finally, the $p$-adic completions \(LD_A(I)_p^\wedge\) and \(D_A(I)_p^\wedge\) are shown to coincide, establishing that derived and classical divided power envelopes agree in the torsion-free, quasi-syntomic setting.

\vspace{0.1in}

For a ring \(R\), let \(\Pair_R^\gamma\) denote the category of divided power \(R\)-algebras \((A,I,\gamma)\). 
A right adjoint to the forgetful functor \(\Pair_R^\gamma \to \Pair_R\) is the \emph{divided power envelope functor}
\begin{equation}
  \label{comp.0.1}
\Env:\Pair_R \to \Pair_R^\gamma,\qquad (A,I) \mapsto (D_A(I),\overline I,\overline\gamma)
\end{equation}
as proved in \cite[Theorem 3.19]{MR491705}. 
For brevity, set
\[
x^{[n]} := \overline\gamma_n(\varphi(x))
\]
for \(x\in I\) and \(n\ge 0\), where \(\varphi:A\to D_A(I)\) is the canonical map. 
Then \(D_A(I)\) is the \(A\)-algebra generated by the elements \(x^{[n]}\) for \(x\in I\) and \(n\ge 0\) \cite[Remarks 3.20(3)]{MR491705}.

There is a natural multiplicative filtration
\begin{equation}
  \label{comp.0.2}
\Fil^{\ge m}D_A(I):=\overline I^{[m]}
\end{equation}
for integers \(m\), where \(\overline I^{[m]}\) is the ideal generated by products \(x_1^{[n_1]}\cdots x_r^{[n_r]}\) with \(x_i\in I\) and \(n_i\ge 0\) such that \(n_1+\cdots+n_r\ge m\). 
The filtered ring \(D_A(I)\) is independent of the base ring \(R\).

Let \((A,I)\in\Pair_{\Z_{(p)}}\). 
If \(D_A(I)\) is torsion-free, then \cite[Corollary 3.23]{MR491705} gives an isomorphism
\[
D_A(I)\otimes_\Z \Q \cong A\otimes_\Z \Q
\]
so we may regard \(D_A(I)\) as a subring of \(A\otimes_\Z \Q\). 
Moreover, for \(x\in I\) and \(n\ge 0\),
\[
x^{[n]}=\frac{x^n}{n!}
\]

\medskip

Before passing to derived divided power algebras, we recall the (non-abelian) left-derived functor construction. 
Let \(\cB\) be an \(\infty\)-category and \(\cC\subset \cP_\Sigma(\cB)\) a full subcategory containing \(\cB\). 
Let \(\cD\) be an \(\infty\)-category admitting sifted colimits. 
For a functor \(F:\cC\to\cD\), the \emph{left derived functor} \(LF:\cP_\Sigma(\cB)\to\cD\) is defined as the left Kan extension of \(F|_{\cB}\) along the inclusion \(\cB\hookrightarrow\cC\). 
By the universal property of left Kan extensions, there is a natural transformation \(LF|_{\cC}\to F\). 
Although the construction depends a priori on \(\cB\), our subsequent choices of \(\cB\) and \(\cC\) will be clear. 
For \(X\in\cC\), we say that \emph{\(F(X)\) is left Kan extended from \(\cB\)} if \(LF(X)\simeq F(X)\).

\medskip

We now recall Mao’s construction of derived divided power algebras. 
Let \(\Pair_R^{\gamma,\st}\) be the essential image of \(\Pair_R^\st\) under \(\Env\). 
Set
\[
\Pair_R^{\gamma,\an}:=\cP_\Sigma(\Pair_R^{\gamma,\st})
\]
By \cite[Lemma 3.19]{MR4851406}, we may view \(\Pair_R^\gamma\) as a full subcategory of \(\Pair_R^{\gamma,\an}\). 
The functor \eqref{comp.0.1} restricts to \(\Pair_R^\st\) and left Kan extends to the \emph{derived divided power envelope}
\[
\Env^\an:\Pair_R^\an\to \Pair_R^{\gamma,\an}
\]
For \((A,I)\in\Pair_R^\an\), let \(LD_I(A)\) be the underlying animated ring of \(\Env^\an(A,I)\), endowed with the natural filtration \(\Fil^{\ge\bullet}LD_I(A)\) obtained by left Kan extension from \(\Pair_R^{\gamma,\st}\). 
This filtered animated ring is independent of \(R\) \cite[Lemma 3.51]{MR4851406}.

\medskip

Next we review the conjugate filtrations on derived divided power algebras; see \cite[Lemma 3.42]{1204.6560} and \cite[Construction 3.66, Definition 3.68]{MR4851406}. 
Let \((B,J)\in\Pair_{\F_p}^\st\), and let \(y_1,\dots,y_n\) be free generators of \(J\). 
Since \(y^p=p!\,y^{[p]}=0\) in \(D_B(J)\), there is a natural \((B/J)^{(1)}\)-algebra structure on \(D_B(J)\). 
For a \(B/J\)-module \(M\), set \(M^{(1)}:=M\otimes_{B,\varphi}B\), where \(\varphi\) is the Frobenius. 
For \(m\in\Z\), let \(\Fil_\conj^{\ge -m}D_B(J)\) be the \(B\)-submodule generated by the monomials \(y_1^{[i_1]}\cdots y_n^{[i_n]}\) with \(i_k\ge 0\) and \(i_1+\cdots+i_n\le m\). 
The filtration \(\Fil_\conj^{\ge\bullet}D_B(J)\) is multiplicative and exhaustive, and \cite[Lemma 3.42]{1204.6560} yields a natural isomorphism
\begin{equation}
\label{comp.10.1}
\gr_\conj^{-m}D_B(J)\cong\bigl(\Gamma_{B/J}^m(J/J^2)\bigr)^{(1)}
\end{equation}

Moreover,
\[
\Gamma_{B/J}^m(J/J^2)\simeq \wedge_{B/J}^m L_{(B/J)/B}[-m]
\]
since \(L_{(B/J)/B}\simeq J/J^2[1]\). 
By left Kan extension, for \((A,I)\in\Pair_{\F_p}^\an\) we obtain a multiplicative exhaustive filtration \(\Fil_\conj^{\ge\bullet}LD_A(I)\) with
\[
\gr_\conj^{-m}LD_A(I)\simeq\bigl(\wedge_{A/I}^m L_{(A/I)/A}\bigr)^{(1)}[-m]
\]
If \(L_{(A/I)/A}\) has Tor amplitude concentrated in degree \(-1\), then \cite[\href{https://stacks.math.columbia.edu/tag/08RA}{Tag 08RA}]{stacks-project} gives \(L_{(A/I)/A}\simeq I/I^2[1]\), so \(I/I^2\) is flat over \(A/I\) and
\begin{equation}
\label{comp.10.2}
\gr_\conj^{-m}LD_A(I)\simeq \bigl(\Gamma_{A/I}^m(I/I^2)\bigr)^{(1)}
\end{equation}

\begin{prop}
\label{comp.12}
Let \((A,I)\in\Pair_R\) with \(R\) a ring. 
If \(LD_A(I)\) is discrete, then the map \(LD_I(A)\to D_I(A)\) is an equivalence of animated rings. 
If, in addition, \(L_{(A/I)/A}\) has Tor amplitude concentrated in degree \(-1\), then \(LD_I(A)\to D_I(A)\) is an equivalence of filtered animated rings.
\end{prop}

\begin{proof}
We assemble results from \cite{MR4851406}. 
The inclusion
\[
\iota:\Pair_R^\gamma\to\Pair_R^{\gamma,\an}
\]
admits a left adjoint
\[
\alpha:\Pair_R^{\gamma,\an}\to\Pair_R^\gamma
\]
\cite[Remark 3.24]{MR4851406}. 
By \cite[Corollary 3.26]{MR4851406},
\[
\alpha(\Env^\an(A,I))\cong \Env(A,I)
\]
If \(LD_A(I)\) is discrete, then \cite[Proposition 3.38]{MR4851406} gives \(\Env^\an(A,I)\simeq \iota(A',I',\gamma')\) for some \((A',I',\gamma')\in\Pair_R^\gamma\). 
Since \(\alpha\iota\simeq\id\), we get \((A',I',\gamma')\simeq \Env(A,I)\), hence \(LD_I(A)\to D_I(A)\) is an equivalence of animated rings. 
Under the Tor-amplitude hypothesis, \cite[Proposition 3.90]{MR4851406} upgrades this to an equivalence of filtered animated rings.
\end{proof}

\begin{prop}[Mathew]
\label{comp.11}
Let \((A,I)\in\Pair_{\Z_{(p)}}\). 
Assume \(A\) and \(A/I\) are torsion-free, \(A/p\) has flat Frobenius, and \(L_{(A/I)/A}\) has \(p\)-complete Tor amplitude concentrated in degree \(-1\). 
Then \(LD_A(I)\to D_A(I)\) is an equivalence of animated rings, and \(D_A(I)\) is torsion-free.
\end{prop}

\begin{proof}
Mathew \cite[Construction 7.15]{zbMATH07729746} sketches the argument; we supply details in the spirit of \cite[Proposition~3.83]{MR4851406}. 
By \cite[Lemma 3.52]{MR4851406} there is a commutative square
\[
\begin{tikzcd}
\Pair_{\Z_{(p)}}^\an \ar[r, "\Env^\an"] \ar[d] & \Pair_{\Z_{(p)}}^{\gamma,\an} \ar[d] \\
\Pair_{\F_p}^\an \ar[r, "\Env^\an"] & \Pair_{\F_p}^{\gamma,\an}
\end{tikzcd}
\]
whose vertical arrows are base change. 
Thus we obtain an equivalence of filtered animated rings
\begin{equation}
\label{comp.11.1}
LD_{A/p}(I/p)\simeq LD_A(I)/p
\end{equation}
where \(I/p\) is the image of \(I\) in \(A/p\). 
By \cite[Corollary 3.79]{MR4851406}, \(LD_{A/p}(I/p)\) is discrete, hence so is \(LD_A(I)/p\). 
Therefore the map \(p:H_i(LD_A(I))\to H_i(LD_A(I))\) is injective for all \(i\ge 0\), so \(H_i(LD_A(I))\) is torsion-free. 
On the other hand, \cite[Lemma 3.47]{MR4851406} gives \(LD_A(I)\otimes_\Z\Q\simeq A\otimes_\Z\Q\), which is discrete; hence \(H_i(LD_A(I))=0\) for \(i\ge 1\). 
Thus \(LD_A(I)\) is discrete and torsion-free, and Proposition \ref{comp.12} applies.
\end{proof}

\begin{prop}
\label{comp.15}
With the notation above, the induced map
\[
LD_A(I)_p^\wedge \to D_A(I)_p^\wedge
\]
is an equivalence of filtered animated rings.
\end{prop}

\begin{proof}
It suffices to show that \(LD_A(I)/p^r \to D_A(I)/p^r\) is an equivalence of filtered complexes for all \(r\ge 1\). 
By induction, it is enough to treat \(r=1\). 
By \cite[\href{https://stacks.math.columbia.edu/tag/07KG}{Tag 07KG(2)}]{stacks-project} there is a ring isomorphism
\[
D_A(I)/p \cong D_{A/p}(I/p)
\]
and, using \eqref{comp.0.2}, this is an isomorphism of filtered rings. 
Together with \eqref{comp.11.1}, we reduce to showing that \(LD_{A/p}(I/p)\to D_{A/p}(I/p)\) is an equivalence of filtered animated rings. 
Since
\(LD_{A/p}(I/p)\) is discrete by \cite[Corollary 3.79]{MR4851406} and
\(L_{(A/I)/A}\) has \(p\)-complete Tor amplitude in degree \(-1\), Proposition \ref{comp.12} applies to \((A/p,I/p)\), and the claim follows.
\end{proof}

\section{Derived log de Rham cohomology}
\label{section:Derived log de Rham cohomology}

In this section we construct and analyze the derived log de Rham complex in the logarithmic setting. 
We begin by recalling the notion of a log connection and its flatness, introducing the Frobenius descent log connection for log $\F_p$-algebras. 
Using the divided power envelope $D_A(I)$, we show that it carries a natural flat log connection compatible with both the Hodge and conjugate filtrations. 
We then define \emph{log Cartier smooth} $\F_p$-algebras as those for which the log cotangent complex is discrete and the inverse Cartier map is an isomorphism, and we verify this property for a large class of examples, including ind-log smooth and log regular objects. 
Under these assumptions, we prove that the complexes $D_A(I)\otimes_A\Omega_{(A,M)/\F_p}^\bullet$ are left Kan extended from the category of standard log pairs. 
Finally, we establish a comparison theorem identifying the $p$-completed derived log de Rham cohomology with the $p$-completed divided power envelope tensored with the logarithmic de Rham complex, generalizing the crystalline–de Rham comparison of \cite{zbMATH07729746} to the logarithmic setting.
\vspace{0.1in}

Let \(R\) be a ring, \((A, M)\) a log \(R\)-algebra, and \(D\) an \(A\)-module.  
A \emph{log connection on \(D\)} is an \(R\)-linear homomorphism
\[
\nabla \colon D \to D \otimes_A \Omega_{(A,M)/R}^1
\]
satisfying the product rule
\[
\nabla(ax) = x \otimes da + a \nabla(x)
\]
for all \(a \in A\) and \(x \in D\).  
The connection \(\nabla\) is called \emph{flat} if the composite
\[
D \xrightarrow{\nabla} D \otimes_A \Omega_{(A,M)/R}^1 \xrightarrow{\nabla} D \otimes_A \Omega_{(A,M)/R}^2
\]
is zero, where the second map is defined by
\[
\nabla(x \otimes b) = x \otimes db + b \wedge \nabla(x)
\]
for \(b \in \Omega_{(A,M)/R}^1\) and \(x \in D\).  
In this case, one obtains the associated complex \(D \otimes_A \Omega_{(A,M)/R}^\bullet\).

\begin{exm}
A connection \(\nabla \colon D \to D \otimes_A \Omega_{A/R}^1\) induces a log connection by composition with the canonical map 
\(D \otimes_A \Omega_{A/R}^1 \to D \otimes_A \Omega_{(A,M)/R}^1\).  
Throughout this paper, we only consider log connections of this form.  
If \(\nabla\) is flat, then so is the induced log connection.
\end{exm}

\begin{df}
Let \((A, M)\) be a log \(\F_p\)-algebra, and let \(S\) denote the source of the Frobenius morphism \(\varphi \colon A \to A\).  
The differential \(d \colon A \to \Omega_{(A,M)/\F_p}^1\) is \(S\)-linear since \(d(x^p y) = x^p dy\) for any \(x, y \in A\).  
For every \(S\)-module \(D\), the \emph{Frobenius descent connection} on \(D \otimes_S A\) is given by
\[
\id \otimes d \colon D \otimes_S A \to D \otimes_S \Omega_{A/\F_p}^1
\]
The \emph{Frobenius descent log connection} on \(D \otimes_S A\) is the composite
\[
D \otimes_S A \xrightarrow{\id \otimes d} D \otimes_S \Omega_{A/\F_p}^1 \to D \otimes_S \Omega_{(A,M)/\F_p}^1
\]
where the second map is canonical.  
Both connections are flat.
\end{df}

\begin{const}
\label{comp.13}
Let \((A,M,I) \in \lPair_{\Z_{(p)}}\), and assume that \(A\) and \(D_A(I)\) are torsion-free.  
Then the differential \(d \colon A \otimes_{\Z} \Q \to \Omega_{(A \otimes_{\Z} \Q, M)/\Z_{(p)}}^1\) restricts to a flat log connection
\[
\nabla \colon D_A(I) \to D_A(I) \otimes_A \Omega_{(A,M)/\Z_{(p)}}^1
\]
satisfying
\[
\nabla\!\left(\frac{y^j}{j!}\right) = \frac{y^{j-1}}{(j-1)!} \, dy
\]
for all \(y \in I\) and \(j \ge 1\).  
The CDGA structure on \(\Omega_{(A \otimes_{\Z} \Q, M)/\Z_{(p)}}^\bullet\) restricts to one on \(D_A(I) \otimes_A \Omega_{(A,M)/\Z_{(p)}}^\bullet\).  
Moreover, the filtration \(\Fil^{\ge \bullet} D_A(I)\) satisfies Griffiths transversality:
\[
\nabla(\Fil^{\ge \bullet} D_A(I)) \subset \Fil^{\ge \bullet-1} D_A(I) \otimes_A \Omega_{(A,M)/\Z_{(p)}}^1
\]
Hence, there is an induced multiplicative filtration on \(\Omega_{(A,M)/\Z_{(p)}}^\bullet \otimes_A D_A(I)\),
\begin{equation}
\label{comp.13.3}
\Fil^{\ge \bullet} D_A(I) \xrightarrow{\nabla} \Fil^{\ge \bullet-1} D_A(I) \otimes_A \Omega_{(A,M)/\Z_{(p)}}^1 \xrightarrow{\nabla} \cdots
\end{equation}
\end{const}

\begin{const}
\label{comp.4}
Let \((B,N,J) \in \Pair_{\F_p}^\text{st}\), and let \(y_1, \ldots, y_n\) be free generators of \(J\).  
There is a flat log connection
\[
\nabla \colon D_B(J) \to D_B(J) \otimes_B \Omega_{(B,N)/\Z_{(p)}}^1
\]
defined by the product rule and by
\[
\nabla(y_i^{[j]}) = y_i^{[j-1]}
\]
for \(1 \le i \le n\) and \(j \ge 1\).  
This log connection is compatible with the conjugate filtration:
\[
\Fil_\conj^{\ge \bullet} D_B(J) \to \Fil_\conj^{\ge \bullet} D_B(J) \otimes_B \Omega_{(B,N)/\Z_{(p)}}^1
\]
Under the isomorphism of \eqref{comp.10.1}, the Frobenius descent log connection
\[
(\Gamma_{B/J}^m(J/J^2))^{(1)} \to (\Gamma_{B/J}^m(J/J^2))^{(1)} \otimes_B \Omega_{(B,N)/\F_p}^1
\]
coincides with the induced log connection
\[
\gr_\conj^{-m} D_B(J) \to \gr_\conj^{-m} D_B(J) \otimes_B \Omega_{(B,N)/\F_p}^1
\]
as in the non-logarithmic case \cite[Lemma 3.44]{1204.6560}.
\end{const}

We now introduce a logarithmic analogue of Cartier smooth \(\F_p\)-algebras \cite[\S 2]{MR4264079}.

\begin{df}
\label{comp.14}
A saturated log \(\F_p\)-algebra \((A,M)\) is \emph{log Cartier smooth} if:
\begin{enumerate}
\item[(i)] The cotangent complex \(L_{(A,M)/\F_p}\) has Tor amplitude concentrated in degree \(0\), hence is discrete and flat.
\item[(ii)] For every \(i\), the inverse Cartier map
\[
C^{-1} \colon \Omega_{(A,M)/\F_p}^i \to H^i(\Omega_{(A,M)/\F_p}^\bullet)
\]
is an isomorphism.
\end{enumerate}
\end{df}

\begin{exm}
If \(A\) is a Cartier smooth \(\F_p\)-algebra, then \((A,0)\) with the trivial monoid is log Cartier smooth.
\end{exm}

\begin{exm}
\label{comp.16}
Let \(k\) be a perfect \(\F_p\)-algebra, and let \((A,M)\) be a saturated log \(\F_p\)-algebra.  
If \((A,M)\) is ind-log smooth, i.e.\ a filtered colimit of log smooth saturated log \(k\)-algebras, then \((A,M)\) is log Cartier smooth.  
The log smooth case follows from \cite[(1.1(iii)), Corollary 8.34]{OlsCot} and \cite[2.12]{MR1293974}, since \(\Omega_{(A,M)/\F_p}^i \cong \Omega_{(A,M)/k}^i\) for all \(i\); the general case follows by passage to colimits.
\end{exm}

\begin{exm}
Let $(A,M)$ be a log regular noetherian $\F_p$-algebra.
By \cite[Corollary 5.2.10(2)]{MR4211089},
there exists an \'etale cover $A\to B$ such that $(B,M)$ is ind-log smooth.
Example \ref{comp.16} implies that $(C,M)$ is log Cartier smooth for any term $C$ appearing in the \v{C}ech nerve of $A\to B$.
Together with fppf descent and the formula
\[
L_{(A,M)/\F_p}\otimes_A C
\simeq
L_{(C,M)/\F_p}
\]
we deduce that $(A,M)$ is log Cartier smooth.
\end{exm}

\begin{rmk}
By \cite[\S 2]{MR4264079}, every valuation ring $R$ of characteristic \(p\) is Cartier smooth.  
We expect that \((R, R \setminus \{0\})\) is log Cartier smooth.
\end{rmk}

\begin{lem}
\label{comp.5}
Let \((A,M)\) be a log \(\F_p\)-algebra, and let \(I \subset A\) be an ideal.  
Assume that \((A,M)\) is log Cartier smooth, \(A\) has flat Frobenius, and \(L_{(A/I)/A}\) has Tor amplitude concentrated in \(-1\).  
Then \(D_A(I) \otimes_A \Omega_{(A,M)/\F_p}^\bullet\) can be left Kan extended from \(\lPair_{\F_p}^\text{st}\).
\end{lem}

\begin{proof}
Applying the left Kan extension from Construction~\ref{comp.4} in \(\lPair_{\Z_{(p)}}^\text{st}\) and using Proposition~\ref{comp.15}, we obtain a filtration \(\Fil_\conj^{\ge \bullet} D_A(I)\) such that the log connection
\[
D_A(I) \to D_A(I) \otimes_A \Omega_{(A,M)/\F_p}^1
\]
is compatible with this filtration.  
From \eqref{comp.10.2} we have isomorphisms
\[
\gr_\conj^{-m} D_A(I) \cong \Gamma_{A/I}^m (I/I^2)^{(1)}
\]
for all \(m\), and under this identification the induced log connection agrees with the Frobenius descent log connection of Construction~\ref{comp.4}.  
It therefore suffices to show that the complexes
\[
\Gamma_{A/I}^m (I/I^2)^{(1)} \otimes_A \Omega_{(A,M)/\F_p}^\bullet
\]
can be left Kan extended from \(\lPair_{\F_p}^\text{st}\).  
Equivalently, for each \(i\),
\begin{equation}
\label{thm.1.1}
H^i\bigl(\Gamma_{A/I}^m (I/I^2)^{(1)} \otimes_A \Omega_{(A,M)/\F_p}^\bullet\bigr)
\end{equation}
can be so extended.  
Since \((A,M)\) is log Cartier smooth, \eqref{thm.1.1} identifies with
\[
\Gamma_{A/I}^m (I/I^2)^{(1)} \otimes_A \Omega_{(A,M)/\F_p}^i
\]
which can indeed be left Kan extended because both factors are.
\end{proof}

\begin{lem}
\label{comp.2}
Let \((B,N,J) \in \lPair_{\Z_{(p)}}^\text{st}\).  
Then there is a natural equivalence of filtered \(\E_\infty\)-rings
\[
D_B(J) \otimes_B \Omega_{(B,N)/\Z_{(p)}}^\bullet \simeq \Omega_{(B/J,N)/\Z_{(p)}}^\bullet
\]
\end{lem}

\begin{proof}
Let \(B'\) be the subring of \(B\) generated by the canonical generators of \(J\), and set \(J' := J \cap B'\).  
By \cite[Lemme V.2.1.2]{MR384804}, there is a natural equivalence
\[
D_{B'}(J') \otimes_{B'} \Omega_{B'/\Z_{(p)}}^\bullet \simeq \Z_{(p)}
\]
Induction on the number of generators of \(J\) shows that this equivalence is compatible with filtrations.  
Arguing as in \cite[Lemme V.2.1.2]{MR384804}, we then obtain a natural equivalence of filtered \(\E_\infty\)-rings
\[
D_B(J) \otimes_B \Omega_{(B,N)/\Z_{(p)}}^\bullet \simeq D_{B'}(J') \otimes_{B'} \Omega_{B'/\Z_{(p)}}^\bullet \otimes_{\Z_{(p)}} \Omega_{(B/J,N)/\Z_{(p)}}^\bullet
\]
which yields the desired result.
\end{proof}

Finally, we extend \cite[Theorem~7.16]{zbMATH07729746}.  
This theorem provides a comparison between log crystalline cohomology and derived log de Rham cohomology, paralleling \cite[Theorem~7.22]{1204.6560}.

\begin{thm}
\label{comp.1}
Let \((A,M)\) be a log \(\Z_{(p)}\)-algebra, and let \(I \subset A\) be an ideal.  
Assume that \(A\) and \(A/I\) are torsion-free, \((A/p,M)\) is log Cartier smooth, \(A/p\) has flat Frobenius, and \(L_{(A/I)/A}\) has \(p\)-complete Tor amplitude concentrated in degree \(-1\).  
Then there is a natural equivalence of filtered \(\E_\infty\)-rings
\[
(L\Omega_{(A/I,M)/\Z_{(p)}})_p^\wedge \simeq D_A(I)_p^\wedge \widehat{\otimes}_A (\Omega_{(A,M)/\Z_{(p)}}^\bullet)_p^\wedge
\]
where the left-hand side carries the Hodge filtration and the right-hand side the filtration induced by \eqref{comp.13.3}.
\end{thm}

\begin{proof}
If \((B,N,J) \in \lPair_{\Z_{(p)}}^\st\), then Lemma~\ref{comp.2} gives a natural equivalence
\[
(D_B(J) \otimes_B \Omega_{(B,J)/\Z_{(p)}}^\bullet)_p^\wedge \simeq (\Omega_{(B/J,N)/\Z_{(p)}}^\bullet)_p^\wedge
\]
It therefore suffices to show that \((D_A(I) \otimes_A \Omega_{(A,M)/\Z_{(p)}}^\bullet)/p^n\) can be left Kan extended from \(\lPair_{\Z_p}^\st\) for all \(n \ge 1\).  
By induction, one reduces to \(n = 1\).  
Since \cite[\href{https://stacks.math.columbia.edu/tag/07KG}{Tag 07KG(2)}]{stacks-project} provides an isomorphism \(D_A(I)/p \cong D_{A/p}(I/p)\), it remains to show that
\[
D_{A/p}(I/p) \otimes_{A/p} \Omega_{(A/p,M)/\F_p}^\bullet
\]
can be left Kan extended from \(\lPair_{\Z_{(p)}}^\st\), which follows from Lemma~\ref{comp.5}.
\end{proof}

\section{Log syntomic cohomology of \texorpdfstring{$(k[x]/x^e,\N)$}{k[x]/xe,N)}}
\label{section:Log syntomic cohomology}

In this section we compute the logarithmic syntomic cohomology of \((k[x]/x^e,\N)\) for a perfect \(\F_p\)-algebra \(k\).
Building on Mathew’s and Sulyma’s results for the non-logarithmic case, we establish an explicit description of the complexes
$$\Z_p^\syn(i)(k[x]/x^e,\N)$$ in terms of Witt vectors, giving an analogue of the classical formulas with logarithmic differentials.
From this computation we deduce a corresponding formula for the homotopy groups of the logarithmic topological cyclic homology \(\logTC(k[x]/x^e,\N)_p^\wedge\), showing that they decompose canonically into Witt components.
\vspace{0.1in}

We then formulate the conjecture that \(\Q_p^\syn(i) := \Z_p^\syn(i)\otimes\Q\) defines a strict cdh sheaf on the category of fs log schemes over \(\F_p\).
Evidence for this conjecture is provided by three results: 
(1) \(\Q_p^\syn(i)\) is already a cdh sheaf on ordinary schemes;
(2) it satisfies nil-invariance in the logarithmic setting, as shown for the log point \((k[x]/x^e,\N)\);
and (3) it obeys strict cdh descent for the basic square relating \((k[x],\N)\), \((k,\N)\), and their underlying schemes.
Finally, using this descent and the analogy with Kato–Nakayama spaces, we show that \(\Q_p^\syn(i)\) is invariant under the inclusion \(\Spec(k,\N)\hookrightarrow \G_m^{\log}\), reflecting a type of homotopy invariance peculiar to log geometry.
\vspace{0.1in}

Throughout this section, let \(k\) be a perfect \(\F_p\)-algebra.
Mathew computed the syntomic cohomology of \(k[x]/x^e\) in the case \(e = 2\) and \(p > 2\) \cite[Theorem~10.4]{zbMATH07729746}, and Sulyma subsequently extended this computation to all values of \(e\) and \(p\) \cite[Theorem~1.1]{MR4632370}.
The following theorem establishes a logarithmic analogue of these results.

\begin{thm}
\label{logfat.1}
For every integer \(e \ge 1\), there is an equivalence of complexes
\begin{equation}
\label{logfat.1.1}
\Z_p^\syn(i)(k[x]/x^e,\N)
\simeq
\begin{cases}
W(k) \oplus W(k)[-1] & \text{if } i = 0 \\[2mm]
\W_{ei-1}(k)[-1] \oplus W(k)[-1] \oplus W(k)[-2] & \text{if } i = 1 \\[2mm]
\W_{ei-1}(k)[-1] & \text{if } i > 1
\end{cases}
\end{equation}
where the log structure homomorphism \(\N \to k[x]/x^e\) sends \(n \mapsto x^n\), and \(\W_j(k)\) denotes the ring of big Witt vectors \cite{MR3316757} truncated at \(\{1,\ldots,j\}\).
\end{thm}

\begin{proof}
Set \(A := W(k)[x]\) and \(I := (x^e)\).
By Theorems~\ref{sat.8} and~\ref{comp.1}, there is an equivalence of filtered complexes
\[
\cPrism_{(k[x]/x^e,\N)} \simeq D_A(I)_p^\wedge \widehat{\otimes}_A (\Omega_{(A,\N)/\Z_{(p)}}^\bullet)_p^\wedge
\]
where the left-hand side carries the Nygaard filtration \(\Fil_\rN^{\ge \bullet}\), and the right-hand side carries the tensor product of the filtration induced by~\eqref{comp.13.3} and the \(p\)-adic filtration on \(\Z_p\).
The right-hand side is the \(p\)-completion of
\[
W(k)\bigl[x, \tfrac{x^e}{1!}, \tfrac{x^{2e}}{2!}, \ldots\bigr]
\xrightarrow{d}
W(k)\bigl[x, \tfrac{x^e}{1!}, \tfrac{x^{2e}}{2!}, \ldots\bigr]\, d\log x
\]
or equivalently of
\[
\bigoplus_{n=0}^\infty W(k)\,\frac{x^n}{\lfloor n/e \rfloor!}
\xrightarrow{d}
\bigoplus_{n=0}^\infty W(k)\,\frac{x^n}{\lfloor n/e \rfloor!}\, d\log x
\]

Let \((\cPrism_{(k[x]/x^e,\N)})_n\) denote the \(n\)-th direct summand.
For each integer \(i \ge 0\), we can express \(\Fil_\rN^{\ge i}\cPrism_{(k[x]/x^e,\N)}\) as the \(p\)-completion of
\[
\bigoplus_{n=0}^\infty p^{\max(i-\lfloor n/e \rfloor,0)}W(k)\,\frac{x^n}{\lfloor n/e \rfloor!}
\xrightarrow{d}
\bigoplus_{n=0}^\infty p^{\max(i-\lfloor n/e \rfloor-1,0)}W(k)\,\frac{x^n}{\lfloor n/e \rfloor!}\, d\log x
\]
Write \((\Fil_\rN^{\ge i}\cPrism_{(k[x]/x^e,\N)})_n\) for its \(n\)-th direct summand.

\paragraph{The case \(n = 0\).}
We compute
\[
\bigl(H^0(\Fil_\rN^{\ge i}\cPrism_{(k[x]/x^e,\N)})_0
\xrightarrow{\varphi/p^i - \can}
H^0(\cPrism_{(k[x]/x^e,\N)})_0\bigr)
\cong
\fib\bigl(W(k) \xrightarrow{1-p^i} W(k)\bigr) 
\]
\[
\simeq
\begin{cases}
W(k) \oplus W(k)[-1] & i = 0 \\[1mm]
0 & i > 0
\end{cases}
\]
and
\[
\bigl(H^1(\Fil_\rN^{\ge i}\cPrism_{(k[x]/x^e,\N)})_0
\xrightarrow{\varphi/p^i - \can}
H^1(\cPrism_{(k[x]/x^e,\N)})_0\bigr)
\simeq
\begin{cases}
0 & i = 0 \\[1mm]
W(k) \oplus W(k)[-1] & i = 1 \\[1mm]
0 & i > 1
\end{cases}
\]
These contribute the \(W(k)\oplus W(k)[-1]\) and \(W(k)[-1]\oplus W(k)[-2]\) summands in~\eqref{logfat.1.1}.

\paragraph{The case \(n > 0\).}
We have
\[
H^1(\cPrism_{(k[x]/x^e,\N)})_n \cong W(k)/n \qquad
H^1(\Fil_\rN^{\ge i}\cPrism_{(k[x]/x^e,\N)})_n \cong
\begin{cases}
W(k)/pn & \lfloor n/e \rfloor < i \\[1mm]
W(k)/n & \lfloor n/e \rfloor \ge i
\end{cases}
\]
Under these identifications,
\[
\can \cong
\begin{cases}
p^{i-\lfloor n/e \rfloor -1} : W(k)/pn \to W(k)/n & \lfloor n/e \rfloor < i \tag{$\ast_1$} \\[1mm]
1 : W(k)/n \to W(k)/n & \lfloor n/e \rfloor \ge i
\end{cases}
\]
and by \cite[Lemma~2.10]{MR4632370},
\[
\varphi/p^i \cong
\begin{cases}
1 : W(k)/pn \to W(k)/pn & \lfloor n/e \rfloor < i \tag{$\ast_2$} \\[1mm]
p^{-i+\lfloor n/e \rfloor +1} : W(k)/n \to W(k)/pn & \lfloor n/e \rfloor \ge i 
\end{cases}
\]

Let \(j\) be a positive integer prime to \(p\), and set
\[
s := \max\bigl(\lceil \log_p(ei/j) \rceil, 0\bigr)
\]
the smallest nonnegative integer such that \(\lfloor p^s j / e \rfloor \ge i\).
Then
\[
\varphi/p^i - \can :
\bigoplus_{m=s}^\infty H^1(\Fil_\rN^{\ge i}\cPrism_{(k[x]/x^e,\N)})_{p^m j}
\to
\bigoplus_{m=s}^\infty H^1(\cPrism_{(k[x]/x^e,\N)})_{p^m j}
\]
is an isomorphism: injectivity follows from \((\ast_1)\), and surjectivity from \((\ast_2)\).
Similarly, using \((\ast_2)\) and \(H^1(\cPrism_{(k[x]/x^e,\N)})_j = 0\), we find that
\[
\varphi/p^i - \can :
\bigoplus_{m=0}^{s-1} H^1(\Fil_\rN^{\ge i}\cPrism_{(k[x]/x^e,\N)})_{p^m j}
\to
\bigoplus_{m=0}^s H^1(\cPrism_{(k[x]/x^e,\N)})_{p^m j}
\]
is an isomorphism.
Hence the kernel of
\[
\varphi/p^i - \can :
\bigoplus_{m=0}^\infty H^1(\Fil_\rN^{\ge i}\cPrism_{(k[x]/x^e,\N)})_{p^m j}
\to
\bigoplus_{m=0}^\infty H^1(\cPrism_{(k[x]/x^e,\N)})_{p^m j}
\]
is \(W(k)/p^{(s-1)+1} \cong W_{\lceil \log_p(ei/j) \rceil}(k)\).
Summing over all \(j\) prime to \(p\) gives the \(p\)-typical decomposition of \(\W_{ei-1}(k)\) \cite[Proposition~1.10]{MR3316757}.
\end{proof}

Recalling \cite[Theorem~1.3, Section~7]{BLPO2}, for a log quasi-regular semiperfectoid \((A,M)\) we define
\[
\Fil_\BMS^{\le i}\logTC(A,M)_p^\wedge := \tau_{\le 2i}\logTC(A,M)_p^\wedge
\]
By log quasi-syntomic descent \cite[Theorem~4.22]{BLPO2}, this yields a complete exhaustive filtration for any log quasi-syntomic \((A,M)\), with graded pieces
\[
\gr_\BMS^i \logTC(A,M)_p^\wedge \simeq \Z_p^\syn(i)(A,M)[2i]
\]

\begin{cor}
\label{logfat.2}
For every integer \(e \ge 1\),
\[
\pi_i\bigl(\logTC(k[x]/x^e,\N)_p^\wedge\bigr) \cong
\begin{cases}
W(k) & i = -1 \\[1mm]
W(k) \oplus W(k) & i = 0 \\[1mm]
W(k) \oplus \W_{e-1}(k) & i = 1 \\[1mm]
\W_{em-1}(k) & i = 2m-1 \ge 3 \\[1mm]
0 & \text{otherwise}
\end{cases}
\]
\end{cor}

\begin{proof}
Theorem~\ref{logfat.1} provides all cases except \(i = 0\), where there is an exact sequence
\[
0 \to
\pi_0\bigl(\Z_p^\syn(1)(k[x]/x^e,\N)[2]\bigr)
\to
\pi_0\bigl(\logTC(k[x]/x^e,\N)_p^\wedge\bigr)
\to
\pi_0\bigl(\Z_p^\syn(0)(k[x]/x^e,\N)\bigr)
\to 0
\]
By \cite[Propositions~6.2, 6.3]{BMS19}, \(\pi_i(\TC(k)_p^\wedge) \cong W(k)\) for \(i = -1,0\) and \(0\) otherwise.
Hence this sequence is one of \(W(k)\)-modules and splits, since \(W(k)\) is free over itself.
\end{proof}

For a topology \(\tau\) on the category of schemes, the \emph{strict \(\tau\)-topology} on fs log schemes is generated by families \(\{f_j : U_j \to X\}_{j \in J}\) with each \(f_j\) strict and \(\{\ul{f_j}\}\) a \(\tau\)-cover.
As recalled in \cite[Section~3]{BLMP}, one can globalize \(\Z_p^\syn(i)\) from log quasi-syntomic log rings to log quasi-syntomic fs \(p\)-adic formal log schemes.
This leads naturally to the following expectation.

\begin{conje}
For every integer \(i\), the complex \(\Q_p^\syn(i) := \Z_p^\syn(i) \otimes \Q\) is a strict cdh sheaf on the category of quasi-compact, quasi-separated fs log schemes over \(\F_p\).
\end{conje}

Let us provide evidence for this conjecture. First, for every integer \(i\),
\(\Q_p^\syn(i)\) is a cdh sheaf on the category of quasi-compact, quasi-separated schemes over \(\F_p\) 
by \cite[Theorem 4.20]{EM}.
Next, we give a nil-invariance example for the log point case.

\begin{prop}
For all integers \(e\ge 1\) and \(i\), the morphism
\[
\Q_p^\syn(i)(k[x]/x^e,\N)\to \Q_p^\syn(i)(k,\N)
\]
naturally induced by the quotient homomorphism \(k[x]/x^e\to k\) is an equivalence.
\end{prop}
\begin{proof}
For every integer \(n\), the proof of Theorem \ref{logfat.1} yields
\begin{align*}
&
\fib\!\Bigl(\Fil_\rN^{\ge i}(\cPrism_{(k[x]/x^e,\N)})_n
\xrightarrow{\varphi-\can}
(\cPrism_{(k[x]/x^e,\N)})_n\Bigr)\otimes \Q
\\[2pt]
\simeq\;&
\begin{cases}
\fib\!\bigl(W(k)\otimes \Q \xrightarrow{0} W(k)\otimes \Q\bigr) & i=0\\
\fib\!\bigl(W(k)\otimes \Q\, d\log x[1] \xrightarrow{0} W(k)\otimes \Q\, d\log x[1]\bigr) & i=1\\
0 & \text{otherwise}
\end{cases}
\end{align*}
Under this identification, the morphism
\(\Q_p^\syn(i)(k[x]/x^e,\N)\to \Q_p^\syn(i)(k[x]/x,\N)\) is the identity.
\end{proof}

Consider the cartesian square of fs log schemes
\[
\begin{tikzcd}
\Spec(k,\N)\ar[d]\ar[r] &
\Spec(k[x],\N)\ar[d,"q"] \\
\Spec(k)\ar[r,"i_0"] &
\Spec(k[x])
\end{tikzcd}
\]
which is a strict cdh square, where \(i_0\) is the zero section, \(q\) forgets the log structure, and the homomorphism \(\N\to k[x]\) sends \(1\mapsto x\). We show that \(\Q_p^\syn(i)\) satisfies descent for this square.

\begin{prop}
\label{logfat.4}
For every integer \(i\), the naturally induced square
\[
\begin{tikzcd}
\Q_p^\syn(i)(k[x])\ar[d]\ar[r] &
\Q_p^\syn(i)(k)\ar[d] \\
\Q_p^\syn(i)(k[x],\N)\ar[r] &
\Q_p^\syn(i)(k,\N)
\end{tikzcd}
\]
is cartesian.
\end{prop}

\begin{proof}
By Theorems \ref{sat.8} and \ref{comp.1}, we may write \(\cPrism_{k[x]}\), \(\cPrism_{(k[x],\N)}\), and \(\cPrism_k\) as the \(p\)-completions of
\begin{gather*}
W(k)[x] \xrightarrow{d} W(k)[x]\, dx \\
W(k)[x] \xrightarrow{d} W(k)[x]\, d\log x \\
W(k)[x,x^2/2!,\ldots] \xrightarrow{d} W(k)[x,x^2/2!,\ldots]\, dx
\end{gather*}
Let \((\cPrism_{k[x]})_n\), \((\cPrism_{(k[x],\N)})_n\), and \((\cPrism_k)_n\) be the degree \(n\) summands, where \(d\log x\) has degree \(0\) and \(dx\) has degree \(1\). Denote the corresponding Nygaard filtrations by \(\Fil_\rN^{\ge \bullet}(\cPrism_{k[x]})_n\), \(\Fil_\rN^{\ge \bullet}(\cPrism_{(k[x],\N)})_n\), and \(\Fil_\rN^{\ge \bullet}(\cPrism_k)_n\). From the proof of Theorem \ref{logfat.1} we have
\begin{align*}
&
\fib\!\Bigl(\Fil_\rN^{\ge i}(\cPrism_{(k,\N)})_n \xrightarrow{\varphi-\can} (\cPrism_{(k,\N)})_n\Bigr)\otimes_\Z \Q
\\[2pt]
\simeq\;&
\begin{cases}
W(k)\otimes \Q \oplus W(k)\otimes \Q[-1] & i=n=0\\
W(k)\otimes \Q[-1] \oplus W(k)\otimes \Q[-2] & i=1,\ n=0\\
0 & \text{otherwise}
\end{cases}
\end{align*}
There is a fiber sequence
\[
\cPrism_{k[x]} \longrightarrow \cPrism_{(k[x],x)} \longrightarrow (0 \to W(k)\, d\log x)
\]
and hence an equivalence
\begin{align*}
&
\fib\!\Bigl(\fib\!\bigl(\Fil_\rN^{\ge i}(\cPrism_{k[x]})_n \xrightarrow{\varphi-\can} (\cPrism_{k[x]})_n\bigr)
\to
\fib\!\bigl(\Fil_\rN^{\ge i}(\cPrism_{(k[x],\N)})_n \xrightarrow{\varphi-\can} (\cPrism_{(k[x],\N)})_n\bigr)\Bigr)
\\[2pt]
\simeq\;&
\begin{cases}
W(k)[-1]\oplus W(k)[-2] & i=1,\ n=0\\
0 & \text{otherwise}
\end{cases}
\end{align*}
A direct computation gives
\[
\fib\!\bigl(\Fil_\rN^{\ge i}(\cPrism_k)_n \xrightarrow{\varphi-\can} (\cPrism_k)_n\bigr)
\simeq
\begin{cases}
W(k)\oplus W(k)[-1] & i=n=0\\
0 & \text{otherwise}
\end{cases}
\]
Combining these identifications and canceling the matching summands proves the claim. 
\end{proof}

The Kato-Nakayama realizations of $\Spec(k,\N)$ and $\G_{m}^{\log}:=(\P^1,0+\infty)$ are homeomorphic to $S^1$ and $S^1 \times [0,1]$ if $k=\C$ as shown in Figure \ref{fig},
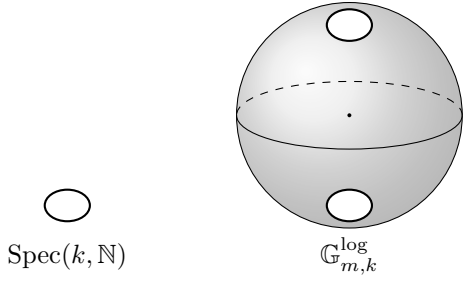
\begin{figure}[tbp]
\centering
\begin{tikzpicture}[scale = 0.75]
  \filldraw[fill=white, thick] (0,-1.6) ellipse (0.4cm and 0.28cm);
  \node at (0,-2.5) {$\Spec(k,\N)$};
\begin{scope}[shift={(5,0)}]
  \shade[ball color = gray!40, opacity = 0.4] (0,0) circle (2cm);
  \draw (0,0) circle (2cm);
  \draw (-2,0) arc (180:360:2 and 0.6);
  \draw[dashed] (2,0) arc (0:180:2 and 0.6);
  \fill[fill=black] (0,0) circle (1pt);
  \filldraw[fill=white, thick] (0,1.6) ellipse (0.4cm and 0.28cm);
  \filldraw[fill=white, thick] (0,-1.6) ellipse (0.4cm and 0.28cm);
    \node at (0,-2.5) {$\G_{m,k}^{\log}$};
\end{scope}
\end{tikzpicture}
\caption{Kato-Nakayama realizations.}
\label{fig}
\end{figure}
which are homotopy equivalent.
This is an ``extra homotopy equivalence'' in logarithmic geometry different from the $\A^1$-homotopy equivalence $\A^1 \to \Spec(k)$.
We show that $\Q_p^\syn(i)$ is invariant under the inclusion $\Spec(k,\N)\to \G_{m}^{\log}$ at zero as follows:

\begin{cor}
\label{logfat.6}
The naturally induced morphism
\[
\Q_p^\syn(i)(\G_m^{\log})\to \Q_p^\syn(i)(k,\N)
\]
is an equivalence for every integer \(i\).
\end{cor}

\begin{proof}
Consider the commutative diagram with cartesian squares
\begin{equation}
\label{logfat.6.1}
\begin{tikzcd}
\Spec(k,\N)\ar[d]\ar[r] &
\Spec(k[x],\N)\ar[d]\ar[r] &
\G_m^{\log}\ar[d] \\
\Spec(k)\ar[r] &
\Spec(k[x])\ar[r] &
\square
\end{tikzcd}
\end{equation}
where \(\square:=(\P^1,\infty)\). The left square is cartesian after applying \(\Q_p^\syn(i)\) by Proposition \ref{logfat.4}. Using the \(\square\)-invariance of \(\Z_p^\syn(i)\) \cite[Corollary 3.11]{BLMP}, it remains to show that the right square \(Q\) is cartesian after applying \(\Q_p^\syn(i)\). By Zariski descent, it suffices to check that
\begin{equation}
\label{logfat.6.2}
Q\times_{\square}\Spec(k[x]),\qquad
Q\times_{\square}\Spec(k[x^{-1}],\N),\qquad
Q\times_{\square}\Spec(k[x,x^{-1}])
\end{equation}
are cartesian after applying \(\Q_p^\syn(i)\). Here the log structure \(\N\to k[x^{-1}]\) sends \(1\mapsto x^{-1}\). In \eqref{logfat.6.2}, the horizontal maps in the first and third squares are isomorphisms, and the vertical maps in the second square are isomorphisms, which proves the claim.
\end{proof}

\section{Representability of log syntomic cohomology in \texorpdfstring{$\logSH(\pt_{\N})$}{logSHptNk}}
\label{section:Representability}

In this section we work over the log point $\pt_{\N}=\Spec (k,\N)$ with $k$ perfect of characteristic $p$. 
We show that several non–$\A^1$-invariant cohomology theories are representable in the logarithmic motivic stable homotopy category $\logSH(\pt_{\N})$ introduced in \cite{BPO2}.
Key results include: the homotopy-cartesian property of basic log squares (Lemma~\ref{represent.1}); 
descent for perfect complexes under dividing covers (Proposition~\ref{represent.2}); 
the identification $L_{(k,\N)/k}\simeq k\oplus k[1]$ (Proposition~\ref{represent.4}); 
and the vanishing $L_{Y/X}\simeq 0$ for log \'etale maps (Proposition~\ref{represent.7}). 
Using these, we prove that the presheaves $X\mapsto R\Gamma_\Zar(X,\wedge^i L_{X/k})$ are log \'etale sheaves and $\square$-invariant (Proposition~\ref{represent.5}), 
which implies that $\logTHH$, $\logTC^-$, $\logTP$, $\logTC$, $\cPrism$, $\Fil_\rN^{\ge i}\cPrism$, and $\Z_p^\syn(i)$ belong to $\logSH_\ketale^\eff(\pt_{\N})$ (Theorem~\ref{represent.6}). 
Finally, these theories admit representing $\P^1$-spectra in $\logSH(\pt_{\N})$
with explicitly computed homotopy groups (Proposition~\ref{logfat.5}).
This will be used in the proof of Theorem \ref{axes.3} below.

\begin{lem}
\label{represent.1}
Let $\N\to P$ be a homomorphism of fs monoids. Consider the induced cartesian square of fs log schemes
\[
Q
:=
\begin{tikzcd}
\A_P\times_{\A_\N}\pt_{\N}\ar[d]\ar[r]&
\pt_{\N}\ar[d]
\\
\A_P\ar[r]&
\A_\N
\end{tikzcd}
\]
Then the square of underlying schemes $\ul{Q}$ is homotopy cartesian in the $\infty$-category of derived schemes.
\end{lem}

\begin{proof}
Let $f\colon k[x]=k[\N]\to k[P]$ be the induced homomorphism. The underlying scheme of $\A_P\times_{\A_\N}\pt_{\N}$ is $\Spec(k[P]/f(x))$. Since $k[P]$ is integral and $f(x)\neq 0$, the element $f(x)$ is regular. It follows that $k[P]/f(x)\simeq k[P]\otimes_{k[x]}^L (k[x]/x)$.
\end{proof}

For a fan $\Sigma$, let $\T_\Sigma$ denote the fs log scheme whose underlying scheme is the associated toric scheme over $k$ with the compactifying log structure coming from the open immersion of the maximal torus.

\begin{prop}
\label{represent.2}
Let $f\colon Y\to X$ be a dividing cover of log smooth fs log schemes over $\pt_\N$. Then for every perfect complex $\cF$ on $\ul{X}$, the unit morphism $\cF\to \ul{f}_*\ul{f}^*\cF$ is an equivalence, where $\ul{f}^*\colon \rD(\Coh(\ul{X}))\to \rD(\Coh(\ul{Y}))$ denotes the pullback between derived $\infty$-categories of coherent sheaves, and $\ul{f}_*$ is its right adjoint.
\end{prop}

\begin{proof}
Since $\cF$ is perfect, the projection formula \cite[\href{https://stacks.math.columbia.edu/tag/0B54}{Tag~0B54}]{stacks-project} yields an equivalence
\[
\ul{f}_*\ul{f}^*\cF \simeq \cF\otimes \ul{f}_*\ul{f}^*\cO_{\ul{X}}
\]
Hence we reduce to the case $\cF=\cO_{\ul{X}}$.

We may work strict \'etale locally on $X$. By \cite[Theorem~IV.3.3.1]{Ogu}, we may assume that $X\to \pt_\N$ admits a chart $\theta\colon \N\to P$ such that $P$ is fs, $\theta$ is injective, the torsion part of $\coker(\theta^\gp)$ is a $p$-torsion-free finite abelian group, and the induced morphism $X\to \A_P\times_{\A_\N}\pt_\N$ is strict \'etale.

By \cite[Proposition~A.11.5]{BPO}, there exists a subdivision $\Sigma$ of the dual cone $P^\vee$ such that the induced morphism
\[
Y\times_{\A_P}\T_\Sigma \longrightarrow X\times_{\A_P}\T_\Sigma
\]
is an isomorphism. By \cite[Lemma~A.11.3]{BPO}, $Y$ admits a Zariski cover $\{Y_i\}_{i\in I}$ such that $Y_i$ is an open subscheme of $X\times_{\A_P}\A_{Q_i}$ for some $P\to Q_i$ with $P^\gp\to Q_i^\gp$ an isomorphism. Let $\Delta_i$ be the subdivision obtained by restricting $\Sigma$ to $Q_i^\vee$. Then
\[
Y_i\times_{\A_{Q_i}}\T_{\Delta_i}\cong Y_i\times_{\A_P}\T_\Sigma
\]
Thus it suffices to show the claim for $X\times_{\A_P}\T_\Sigma\to X$ and for each $Y_i\times_{\A_{Q_i}}\T_{\Delta_i}\to Y_i$. Up to renaming, we reduce to $X\times_{\A_P}\T_\Sigma\to X$.

Since $\A_P\times_{\A_\N}\pt_\N$ is log smooth over $\pt_\N$ by \cite[Theorem~IV.3.1.8]{Ogu}, flat base change \cite[Proposition~2.5.4.5]{SAG} reduces us to $X=\A_P\times_{\A_\N}\pt_\N$.

Consider the commutative diagram of schemes with cartesian square
\[
\begin{tikzcd}
\ul{X}\times_{\ul{\A_P}}\ul{\T_\Sigma}\ar[d]\ar[r,"\ul{f}"]&
\ul{X}\ar[d,"\ul{u}"]\ar[r]&
\{0\}\ar[d]
\\
\ul{\T_\Sigma}\ar[r,"\ul{g}"]&
\ul{\A_P}\ar[r]&
\A^1
\end{tikzcd}
\]
The right and outer squares are homotopy cartesian by Lemma \ref{represent.1}, so the left square is homotopy cartesian as well. As in \cite[Proposition~9.2.4]{BPO}, the unit $\cO_{\A_P}\to \ul{g}_*\ul{g}^*\cO_{\A_P}$ is an equivalence. Applying $\ul{u}^*$ and using base change for derived schemes \cite[Proposition~2.5.4.5]{SAG} finishes the proof.
\end{proof}

\begin{rmk}
There is a gap in \cite[Proof of Theorem~6.16]{BLPO}: the argument fails if the square $Q$ in loc.\ cit.\ is not homotopy cartesian.
\end{rmk}

\begin{prop}
\label{represent.4}
We have $L_{(k,\N)/k}\simeq k \oplus k[1]$. In particular, $L_{(k,\N)/k}$ is perfect.
\end{prop}

\begin{proof}
Using \cite[(3.6)]{BLPO}, we see that $L_{(k,\N)/k}$ is equivalent to the pushout of
\[
L_{k/k} \leftarrow k\otimes_{k[\N]} L_{k[\N]/k} \to k\otimes_{k[\N]} L_{(k[\N],\N)/k}
\]
This identifies with
\[
\cofib\!\bigl(k\otimes_{k[x]} k[x]\; dx \longrightarrow k\otimes_{k[x]} k[x]\; d\log x\bigr),
\]
which is $\cofib(k\xrightarrow{0} k)$.
\end{proof}

Let $X$ be a log smooth fs log scheme over $\pt_\N$. Then $X$ is integral over $\pt_{\N}$ by \cite[Proposition~III.2.5.3(3)]{Ogu}. Hence, by \cite[Corollary~8.34]{OlsCot}, we have
\begin{equation}
\label{represent.7.1}
L_{X/\pt_\N}\simeq \Omega_{X/\pt_\N}^1
\end{equation}

\begin{prop}
\label{represent.7}
Let $f\colon Y\to X$ be a log \'etale morphism of log smooth fs log schemes over $\pt_\N$. Then $L_{Y/X}\simeq 0$.
\end{prop}

\begin{proof}
From the transitivity fiber sequence
\[
f^*L_{X/\pt_\N} \longrightarrow L_{Y/\pt_\N} \longrightarrow L_{Y/X},
\]
it suffices to show $f^*L_{X/\pt_\N} \simeq L_{Y/\pt_\N}$. We have $L_{X/\pt_\N}\simeq \Omega_{X/\pt_\N}^1$ and $L_{Y/\pt_\N}\simeq \Omega_{Y/\pt_\N}^1$ by \eqref{represent.7.1}. Since $f$ is log \'etale, $\Omega_{Y/X}^1\cong 0$. Using the exact sequence
\[
0 \to f^*\Omega_{X/\pt_{\N}}^1 \to \Omega_{Y/\pt_{\N}}^1 \to \Omega_{Y/X}^1 \to 0
\]
finishes the proof.
\end{proof}

As in \cite{BPO2}, for a quasi-compact quasi-separated scheme $S$, let $\logSH^\eff(S)$ (resp.\ $\logSH_\ketale^\eff(S)$) denote the $\infty$-category of $\square$-invariant dividing Nisnevich (resp.\ log \'etale) sheaves of spectra on the category of log smooth fs log schemes over $S$. Here we work with fs log schemes with \'etale log structures, while \cite{BPO2} is written for Zariski log structures. The construction is independent of this choice, since every fs log scheme $X$ with an \'etale log structure admits a log blow-up $Y\to X$ such that $Y$ carries a Zariski log structure \cite[Theorem~5.4]{Niz}.

\begin{prop}
\label{represent.5}
The presheaf $X\mapsto R\Gamma_\Zar(X,\wedge^i L_{X/k})$ on the category of log smooth fs log schemes over $\pt_\N$ is a log \'etale sheaf and is $\square$-invariant for every integer $i$. Hence
\[
R\Gamma_\Zar(-,\wedge^i L_{-/k})\in \logSH_\ketale^\eff(\pt_\N)
\]
\end{prop}

\begin{proof}
Let $f\colon Y\to X$ be a dividing cover of log smooth fs log schemes over $\pt_\N$. By Proposition \ref{represent.7} and the transitivity fiber sequence
\[
f^*L_{X/k}\to L_{Y/k} \to L_{Y/X},
\]
we have $f^*L_{X/k}\simeq L_{Y/k}$. Let $g\colon X\to \pt_\N$ be the structure morphism. Together with \eqref{represent.7.1}, Proposition \ref{represent.4}, and the transitivity fiber sequence
\begin{equation}
\label{represent.5.1}
g^*L_{\pt_\N/k}\to L_{X/k}\to L_{X/\pt_\N},
\end{equation}
we deduce that $L_{X/k}$ is perfect. Hence Proposition \ref{represent.2} implies
\[
R\Gamma_\Zar(X,\wedge^i L_{X/k}) \simeq R\Gamma_\Zar\!\bigl(Y, f^*(\wedge^i L_{X/k})\bigr)
\]
Since $f^*L_{X/k}\simeq L_{Y/k}$, the presheaf $X\mapsto R\Gamma_\Zar(X,\wedge^i L_{X/k})$ is dividing invariant.

Every log \'etale covering $\{V_i\to V\}$ admits a refinement $\{W_i\to W\}$ with $W\to V$ dividing and $\{W_i\to W\}$ Kummer \'etale \cite[Proposition~3.9]{MR3658728}. Hence it remains to show that the presheaf $X\mapsto R\Gamma_\Zar(X,\wedge^i L_{X/k})$ is a Kummer \'etale sheaf and $\square$-invariant.

Applying $\wedge^i$ to \eqref{represent.5.1} gives a finite filtration on $\wedge^i L_{X/k}$ whose $d$-th graded piece is $\wedge^d g^*L_{\pt_\N/k}\otimes \wedge^{i-d}L_{X/\pt_\N}$. Using \eqref{represent.7.1} and Proposition \ref{represent.4}, it suffices to show that
\[
X\longmapsto R\Gamma_\Zar\!\bigl(X, \wedge^d(\cO_X\oplus \cO_X[1])\otimes \Omega^{i-d}_{X/\pt_{\N}}\bigr)
\]
is a Kummer \'etale sheaf and $\square$-invariant. For any Kummer \'etale $g\colon X'\to X$ we have $g^*\Omega_{X/\pt_{\N}}^1\cong \Omega_{X'/\pt_{\N}}^1$. The claim then follows from \cite[Proposition~3.27]{MR2452875} and \cite[Proposition~8.3]{BLPO}.
\end{proof}

\begin{thm}
\label{represent.6}
The presheaves $\logTHH$, $\logTC^-$, $\logTP$, $\logTC$, $\cPrism$, $\Fil_\rN^{\geq i}\cPrism$, and $\Z_p^\syn(i)$ on the category of log smooth fs log schemes over $\pt_\N$ are log \'etale sheaves and are $\square$-invariant for every integer $i$. Hence
\[
\logTHH,\ \logTC^-,\ \logTP,\ \logTC,\ \cPrism,\ \Fil_\rN^{\geq i}\cPrism,\ \Z_p^\syn(i)
\in \logSH_\ketale^\eff(\pt_\N)
\]
\end{thm}

\begin{proof}
Argue as in \cite[Corollary~3.4]{BMS19} to deduce that $\logTHH$, $\logTC^-$, $\logTP$, and $\logTC$ are log \'etale sheaves from Proposition \ref{represent.5}. These are also $\square$-invariant by \cite[Theorem~8.4.4]{BPO2}.

By the definition of $\Z_p^\syn(i)$, it remains to treat $\Fil_\rN^{\geq i}\cPrism$. The Nygaard filtration is complete and exhaustive, so it suffices to show that $\gr_\rN^i\cPrism$ is a log \'etale sheaf and $\square$-invariant. Using the finite filtration on $\gr_\rN^i\cPrism$ from \cite[Proposition~7.4]{BLPO2}, we reduce to showing that $\wedge^i L_{-/k}$ is a log \'etale sheaf and $\square$-invariant, which follows from Proposition \ref{represent.5}.
\end{proof}

Following \cite{BPO2}, for a quasi-compact quasi-separated scheme $S$, write $\logSH(S)$ (resp.\ $\logSH_\ketale(S)$) for the $\infty$-category of $\P^1$-spectra in $\logSH^\eff(S)$ (resp.\ $\logSH_\ketale^\eff(S)$). Let $\unit$ denote the unit object of these symmetric monoidal $\infty$-categories.

Let $\map$ denote the mapping spectrum. As in \cite[Definition~8.5.3]{BPO2}, using the projective bundle formula for $\P^1$ \cite[Theorem~8.5.1]{BPO3} as bonding maps we obtain
\[
\blogTHH := (\logTHH,\logTHH,\ldots)\in \logSH(\pt_\N),
\]
satisfying
\[
\map_{\logSH(\pt_\N)}(\Sigma_{\P^1}^\infty X_+,\blogTHH)\simeq \logTC(X)
\]
for every log smooth fs log scheme $X$ over $\pt_\N$, and
\[
\blogTHH \simeq \Sigma^{2,1}\blogTHH
\]
Similarly, 
in the spirit of \cite[Definition~8.5.3]{BPO2} and \cite[Definition~5.12]{BPO3}, 
we define
\[
\blogTC^-,\ \blogTP,\ \blogTC,\ \bPrism,\ \Fil_\rN^{\geq i}\bPrism,\ \MZ_p^\syn(i)\in \logSH(\pt_\N)
\]

\begin{prop}
\label{logfat.5}
For all integers $e\geq 1$ and $i$,
\[
\pi_{i,j}(\blogTC(\pt_\N)_p^\wedge)
:=
\pi_i\!\left(\map_{\logSH(\pt_\N)}(\unit,\Sigma^{0,j}\blogTC)_p^\wedge\right)
\]
\[
\cong
\begin{cases}
W(k) & i-2j=-1,\\
W(k)\oplus W(k) & i-2j=0,\\
W(k) & i-2j=1,\\
\W_{m-1}(k) & i-2j=2m-1\geq 3,\\
0 & \text{otherwise.}
\end{cases}
\]
\end{prop}

\begin{proof}
Immediate from Corollary \ref{logfat.2}.
\end{proof}

\section{Log syntomic cohomology of projective log coordinate axes}

This section computes the logarithmic syntomic cohomology of the projective coordinate axes 
\( D = (\P^2, \ul{D}) \times_{\P^2} \ul{D} \), 
where \(\ul{D}\) is the union of the coordinate lines in \(\P^2\). 
Using toric geometry, 
\(D\) is expressed as a log fiber product of toric varieties, allowing a description in terms of fans and cones. 
A key step is the identification 
\[
\Sigma_+^\infty
\T_{\langle \sigma,\tau,\tau'\rangle,[e_1]+[e_2]}
\simeq
\unit \oplus \Sigma^{2,1}\unit
\]
in \(\logSH^\eff(\A_\N)\), 
established through Zariski distinguished squares and dividing covers. 
This leads to the main result
\[
\Z_p^\syn(i)(D)
\simeq
\Z_p^\syn(i)(k,\N)
\oplus
\Z_p^\syn(i-1)(k,\N)[-2],
\]
relating the syntomic cohomology of \(D\) to that of the logarithmic point. 
Together with Theorem~\ref{logfat.1}, this gives a complete computation of \(\Z_p^\syn(i)(D)\).
\vspace{0.1in}

Let $k$ be a perfect $\F_p$-algebra.
Consider the projective coordinate axes in $\P^2$ given by
\[
\ul{D}
:=
\{[x:y:z]:
x=0
\text{ or }
y=0
\}
\]
We have the associated log scheme $(\P^2,\ul{D})$ whose underlying scheme is $\P^2$ with the compactifying log structure associated with the open immersion $\P^2-\ul{D}\to \P^2$
and we set
\[
D:=(\P^2,\ul{D})\times_{\P^2}\ul{D}
\]
This is a basic example of a semistable fs log scheme.
The goal of this section is to compute $\Z_p^\syn(i)(D)$ for each integer $i$.

To prepare for later arguments, we introduce notation from toric geometry.
For cones $\sigma_1,\ldots,\sigma_n$ in a lattice $N$, let $\langle \sigma_1,\ldots,\sigma_n\rangle$ denote the smallest fan in $N$ containing $\sigma_1,\ldots,\sigma_n$.
For a smooth fan $\Sigma$ and rays $r_1,\ldots,r_d\in \Sigma$, let $\T_{\Sigma,[r_1]+\cdots+[r_d]}$ denote the fs log scheme whose underlying scheme is the associated toric scheme over $k$ with the Deligne–Faltings log structure \cite[Section III.1.7]{Ogu} associated with the Cartier divisors induced by the rays $r_1,\ldots,r_d$.
If $r_1,\ldots,r_d$ are all the rays of $\Sigma$, then $\T_{\Sigma,[r_1]+\cdots+[r_d]}=\T_\Sigma$.

Let $e_1$ and $e_2$ be the standard coordinates in $\N^2$.
Consider the cones in $\N^2$
\begin{gather*}
\sigma:=\Cone(e_1,e_2) \\
\tau:=\Cone(e_1,e_1-e_2) \\
\tau':=\Cone(e_2,e_2-e_1)
\end{gather*}
There is an isomorphism
\[
D
\cong
\T_{\langle \sigma,\tau,\tau' \rangle,[e_1]+[e_2]}
\times_{\A_\N}
\pt_\N
\]
where $\T_{\langle \sigma,\tau,\tau' \rangle,[e_1]+[e_2]}\to \A_\N$ is the log smooth morphism induced by the summation homomorphism $\N^2\to \N$.
This isomorphism follows from the local descriptions
\begin{gather*}
\T_{\langle \sigma \rangle,[e_1]+[e_2]} \times_{\A_\N}\pt_\N
\cong
\Spec(k[x,y]/(xy),\N x\oplus \N y) \\
\T_{\langle \tau\rangle,[e_1]}\times_{\A_\N} \pt_\N
\cong
\Spec(k[xy,y^{-1}]/(xy),\N (xy)) \\
\T_{\langle \tau'\rangle,[e_2]}\times_{\A_\N} \pt_\N
\cong
\Spec(k[xy,x^{-1}]/(xy),\N (xy))
\end{gather*}
For a ring $R$ and an element $a\in R$, the log ring $(R,\N a)$ means that the structure homomorphism $\N\to R$ sends $1$ to $a$. 
We use a similar convention for $(R,\N a\oplus \N b)$ with $a,b\in R$.

\begin{prop}
\label{axes.4}
There is an equivalence
\[
\Sigma_+^\infty
\T_{\langle \sigma,\tau,\tau'\rangle,[e_1]+[e_2]}
\simeq
\unit \oplus \Sigma^{2,1}\unit
\]
in $\logSH^\eff(\A_\N)$.
\end{prop}

\begin{proof}
It suffices to show that the induced morphism
\[
\Sigma_+^\infty \T_{\langle \sigma,\tau,\tau' \rangle,[e_1]+[e_2]}
\to
\Sigma_+^\infty \T_{\langle \sigma\cup \tau,\tau'\rangle,[e_2]}
\]
is an equivalence, since $\T_{\langle \sigma\cup \tau,\tau'\rangle,[e_2]}\cong \P^1 \times \A_\N$.
The squares
\[
\begin{tikzcd}[column sep=small]
\T_{\langle \sigma\cap \tau'\rangle,[e_2]}\ar[r]\ar[d]&
\T_{\langle \sigma,\tau\rangle,[e_1]+[e_2]}\ar[d] \\
\T_{\langle \tau' \rangle,[e_2]}\ar[r]&
\T_{\langle \sigma,\tau,\tau' \rangle,[e_1]+[e_2]}
\end{tikzcd}
\qquad
\begin{tikzcd}[column sep=small]
\T_{\langle \sigma\cap \tau'\rangle,[e_2]}\ar[r]\ar[d]&
\T_{\langle \sigma\cup \tau\rangle,[e_2]}\ar[d] \\
\T_{\langle \tau' \rangle,[e_2]}\ar[r]&
\T_{\langle \sigma\cup \tau,\tau' \rangle,[e_2]}
\end{tikzcd}
\]
are Zariski distinguished squares.
Hence it suffices to show that the induced morphism
\[
\Sigma_+^\infty \T_{\langle \sigma,\tau \rangle,[e_1]+[e_2]}
\to
\Sigma_+^\infty \T_{\langle \sigma\cup \tau\rangle,[e_2]}
\]
is an equivalence.
The squares
\[
\begin{tikzcd}[column sep=small]
\T_{\langle \sigma\cap \tau'\rangle,[e_2]}\ar[r]\ar[d]&
\T_{\langle \sigma,\tau\rangle,[e_1]+[e_2]}\ar[d] \\
\T_{\langle \tau' \rangle,[e_2]+[v]}\ar[r]&
\T_{\langle \sigma,\tau,\tau' \rangle,[e_1]+[e_2]+[v]}
\end{tikzcd}
\qquad
\begin{tikzcd}[column sep=small]
\T_{\langle \sigma\cap \tau'\rangle,[e_2]}\ar[r]\ar[d]&
\T_{\langle \sigma\cup \tau\rangle,[e_2]}\ar[d] \\
\T_{\langle \tau' \rangle,[e_2]+[v]}\ar[r]&
\T_{\langle \sigma\cup \tau,\tau' \rangle,[e_2]+[v]}
\end{tikzcd}
\]
with $v:=-e_1+e_2$ are also Zariski distinguished squares.
Thus it suffices to show that the induced morphism
\[
\Sigma_+^\infty \T_{\langle \sigma,\tau,\tau' \rangle,[e_1]+[e_2]+[v]}
\to
\Sigma_+^\infty \T_{\langle \sigma\cup \tau,\tau'\rangle,[e_2]+[v]}
\]
is an equivalence.

We have equivalences
\[
\Sigma_+^\infty
\T_{\langle \sigma \cup \tau,\tau'\rangle,[e_2]+[v]}
\simeq
\unit
\simeq
\Sigma_+^\infty
\T_{\langle \sigma\cup \tau',\tau\rangle,[e_1]+[v]}
\]
since both $\T_{\langle \sigma \cup \tau,\tau'\rangle,[e_2]+[v]}$ and $\T_{\langle \sigma\cup \tau',\tau\rangle,[e_1]+[v]}$ are $\square$-bundles over $\A_\N$.
Hence it suffices to show that the induced morphism
\[
\Sigma_+^\infty \T_{\langle \sigma,\tau,\tau' \rangle,[e_1]+[e_2]+[v]}
\to
\Sigma_+^\infty \T_{\langle \sigma\cup \tau',\tau\rangle,[e_1]+[v]}
\]
is an equivalence.
This follows since the morphism
\[
\T_{\langle \sigma,\tau,\tau' \rangle,[e_1]+[e_2]+[v]}
\to
\T_{\langle \sigma\cup \tau',\tau\rangle,[e_1]+[v]}
\]
is a dividing cover.
\end{proof}

\begin{thm}
\label{axes.3}
For the log scheme $D$ defined above, there is an equivalence of complexes
\[
\Z_p^\syn(i)(D)
\simeq
\Z_p^\syn(i)(k,\N)
\oplus
\Z_p^\syn(i-1)(k,\N)[-2]
\]
for every integer $i$.
\end{thm}

\begin{proof}
Using the pullback functor $\logSH^\eff(\A_\N)\to \logSH^\eff(\pt_\N)$,
Proposition~\ref{axes.4} gives an equivalence
\[
\Sigma_+^\infty D
\simeq
\unit \oplus \Sigma^{2,1}\unit
\]
in $\logSH^\eff(\pt_\N)$.
Combining this with the projective bundle formula for $\P^1$ in \cite[Construction~5.9]{BPO3} yields the claim.
\end{proof}

Combining this result with Theorem~\ref{logfat.1} determines $\Z_p^\syn(i)(D)$.

\bibliography{references}
\bibliographystyle{siam}

\end{document}